\documentclass[11pt,oneside,english]{amsart}
\RequirePackage{amsthm,amsmath,mathtools}

\usepackage[T1]{fontenc}
\usepackage{geometry}
\usepackage{caption,subcaption,multirow,enumitem}
\usepackage{scalefnt}
\usepackage[english]{babel}
\usepackage{float,graphicx,verbatim}
\usepackage{amstext}
\usepackage{amssymb,mathrsfs,stmaryrd}
\usepackage[numbers]{natbib}
\RequirePackage[colorlinks,citecolor=blue,urlcolor=blue]{hyperref}

\RequirePackage{amsmath,mathtools}%,amsthm

\numberwithin{equation}{section}
\numberwithin{figure}{section}
\theoremstyle{plain}

\theoremstyle{plain}
\newtheorem*{lyxalgorithm*}{\protect\algorithmname}
\theoremstyle{plain}
\newtheorem{thm}{\protect\theoremname}
\theoremstyle{plain}
\newtheorem*{asm*}{\protect\assumptionname}
\theoremstyle{plain}

\theoremstyle{plain}
\newtheorem{lem}[thm]{\protect\lemmaname}
\theoremstyle{plain}
\newtheorem{cor}[thm]{\protect\corollaryname}
\theoremstyle{plain}
\newtheorem{prop}[thm]{\protect\propositionname}
\theoremstyle{definition}
\newtheorem*{example*}{\protect\examplename}

\usepackage{tikz}
\usetikzlibrary{calc}

\usepackage{algorithm,algpseudocode,bbm}%,multirow}
\newcommand{\1}{\mathbbm{1}}
\newcommand{\eqd}{\overset{d}{=}}

\newcommand{\C}{\mathbb{C}}

\newcommand{\e}{\mathbb{E}}
\newcommand{\E}{\mathrm{e}}
\newcommand{\D}{\mathrm{d}}

\newcommand{\p}{\mathbb{P}}

\newcommand{\Z}{\mathbb{Z}}
\newcommand{\R}{\mathbb{R}}
\newcommand{\N}{\mathbb{N}}
\newcommand{\Dir}{\mathrm{Dir}}
\newcommand{\ti}[1]{\tilde{#1}}
\newcommand{\ov}[1]{\overline{#1}}
\newcommand{\un}[1]{\underline{#1}}
\newcommand{\fl}[1]{\left\lfloor{#1}\right\rfloor}
\newcommand{\cl}[1]{\left\lceil{#1}\right\rceil}
\newcommand{\nf}[1]{{\normalfont#1}}

\newcommand{\da}{\downarrow}

\newcommand{\cip}{\overset{\p}{\to}}

\newcommand{\ve}{\varepsilon}

\newcommand{\sgn}{\text{\normalfont sgn}}

\newcommand{\Leb}{\text{\normalfont Leb}}
\newcommand{\cadlag}{\text{\normalfont c\`adl\`ag}}

\newcommand{\CM}[2]{C_{#1}^{#2}}

\newcommand{\br}[1]{\llbracket{#1}\rrbracket}
\newcommand{\wt}[1]{\widetilde{#1}}
\newcommand{\Mod}[1]{\,\mathrm{mod}\,#1}

\newcommand{\U}{\text{\normalfont U}}

\newcommand{\Exp}{\text{\normalfont Exp}}

\usepackage{pgfplots} % Graphs

\makeatother
  \providecommand{\algorithmname}{Algorithm}
  \providecommand{\assumptionname}{Assumption}
  \providecommand{\examplename}{Example}
  \providecommand{\lemmaname}{Lemma}
  \providecommand{\propositionname}{Proposition}
  \providecommand{\remarkname}{Remark}
\providecommand{\corollaryname}{Corollary}
\providecommand{\theoremname}{Theorem}

\makeatletter
\@namedef{subjclassname@2020}{%
	\textup{2020} Mathematics Subject Classification}
\makeatother

\begin{document}

%\title[Fluctuation theory for L\'evy %processes]{A characterisation of the %convex minorant of a L\'evy process}

\title[Convex minorants \& the fluctuations of L\'evy processes]{Convex minorants and the fluctuation theory of L\'evy processes}

\author[Gonz\'alez C\'azares and Mijatovi\'c]{Jorge Gonz\'{a}lez C\'{a}zares \and Aleksandar Mijatovi\'{c}}

\address{Department of Statistics, University of Warwick \and 
	The Alan Turing Institute, UK}

\email{jorge.gonzalez-cazares@warwick.ac.uk}
\email{a.mijatovic@warwick.ac.uk}

\begin{abstract}
We establish a novel characterisation of the law of the convex minorant of any L\'evy process. Our  self-contained elementary proof is based on the analysis of piecewise linear convex functions and requires only very basic properties of L\'evy processes. Our main result provides a new simple and self-contained approach to the fluctuation theory of L\'evy processes, circumventing  local time and excursion theory. Easy corollaries include classical theorems, such as  Rogozin's regularity criterion, Spitzer's identities and the Wiener-Hopf factorisation, as well as a novel factorisation identity.
\end{abstract}


\maketitle
\section{Introduction}
This paper provides elementary, self-contained proofs of some of the main results of the fluctuation theory for L\'evy processes, a subject of classical interest in probability (see e.g. monographs~\cite[Ch.VI]{MR1406564}, \cite[Ch.6]{MR2250061}, \cite[Ch.9]{MR3185174} and the references therein). Our approach is based on a novel  
characterisation of 
the law of the convex minorant of a path of \textit{any} L\'evy process, given in our main result  (Theorem~\ref{thm:levy-minorant} in Section~\ref{sec:convex-minorants}) below.
Its direct consequence %of Theorem~\ref{thm:levy-minorant}
is a characterisation 
in Theorem~\ref{thm:SB-representation} below
of the law of the triplet of the supremum,
%$\ov X_T$, 
the time
%$\tau_T(X)$ 
the supremum was attained and the position at $T$ %$X_T$
of the L\'evy process $X$ on the time interval $[0,T]$ of fixed length $T$. In Section~\ref{sec:SB-representation} we use Theorem~\ref{thm:SB-representation} to provide short proofs of the Rogozin's criterion for the regularity of $X$ at its starting point,
Spitzer's formula for the supremum of $X$, the Wiener-Hopf identities and the continuity of the law of the triplet. All these results are easy corollaries of Theorem~\ref{thm:SB-representation}
and basic properties of $X$. In particular, our approach circumvents local times and excursion theory used in other probabilistic proofs of fluctuation identities~\cite{MR1406564, MR2250061} and the continuity of the law of the triplet~\cite{MR3098676}. To demonstrate further the usefulness of 
Theorem~\ref{thm:levy-minorant}, we conclude Section~\ref{sec:SB-representation} by showing that
the vertex process of a L\'evy process (introduced in~\cite{MR714964} for Brownian motion and later studied in~\cite[Ch.XI.1]{MR1739699}, both on infinite time horizon)
has independent increments. This result 
characterises the joint law of an $n$-tuple of the suprema of the L\'evy process $X$ perturbed by $n$ different drifts, 
see Corollary~\ref{cor:multiple_drifts} below, yielding what appears to be a novel generalisation of the Wiener-Hopf factorisation. 
%This result follows easily from %Theorem~\ref{thm:SB-representation} and has, to %the best of our knowledge, not been established %before.

In Section~\ref{sec:convex-minorants} we provide an elementary proof of Theorem~\ref{thm:levy-minorant}, characterising the law of the convex minorant $C^X_T$ of the path of the L\'evy process $X$ on the time interval $[0,T]$.
Theorem~\ref{thm:levy-minorant} provides an explicit construction of a random piecewise linear convex function, whose law equals that of $C_T^X$. 
Our main result can be viewed as a generalisation of~\cite[Thm~1]{MR2978134}, where such a characterisation was obtained for L\'evy processes with diffuse marginals. The main technical step in~\cite{MR2978134} is a limit result on the Skorkhod space establishing that the faces of the convex minorant of the random-walk skeleton of $X$, described by~\cite[Thm~1]{MR2825583}, converge to the faces of $C_T^X$. This is a highly non-trivial result because it requires the control of the convergence of the excursions of the approximating random-walk skeletons to the excursions of $X$ away from the faces of $C_T^X$. The less activity $X$ has, the harder the convergence argument in~\cite{MR2978134} becomes, with  the conclusion of~\cite[Thm~1]{MR2978134} not being true 
(by Theorem~\ref{thm:levy-minorant})
when the marginals of $X$ have atoms. 

In contrast, the proof of Theorem~\ref{thm:levy-minorant} in Section~\ref{sec:convex-minorants} below  focuses on the convergence of the entire convex minorant of the skeleton of $X$, rather than each face separately. This approach circumvents the delicate face-by-face convergence in the Skorkhod space, which does not hold in general. 
%establishing Theorem~\ref{thm:levy-minorant} %for all L\'evy processes. 
Our proof relies entirely on elementary geometrical arguments in Section~\ref{subsec:CM-PL} to control the convergence of the piecewise linear convex functions of the approximating random walks to the convex function whose law is equal to that of $C_T^X$ for \textit{all} L\'evy processes. Considering the convex minorant as a whole, rather than face-by-face, was crucial both in finding the correct formulation of Theorem~\ref{thm:levy-minorant},
which generalises~\cite[Thm~1]{MR2978134}, and in its eventual proof. In this paper we give a complete self-contained account of the proof of Theorem~\ref{thm:levy-minorant}, which perhaps surprisingly requires only very basic properties of L\'evy processes
(see Section~\ref{subsec:main_theorem} and a YouTube  \href{https://youtu.be/hEg4YmxOgXA}{presentation}~\cite{Presentation_AM} for an overview of the proof).

%As our aim here is to provide a self-contained %simple proof of Theorem~\ref{thm:levy-minorant} %and all its corollaries, in %Section~\ref{subsec:CM-RW} we recall
%from~\cite{MR2825583} the 3214 transformation %for random walks and give a proof %of~\cite[Thm~1]{MR2825583}.  

\section{Fluctuations of L\'evy processes: the stick-breaking approach}
\label{sec:SB-representation}

Let $x:[0,T]\to\R$ be a \cadlag~path  
(i.e. right-continuous with left limits),  $\ov x:[0,T]\to\R$ 
be its running supremum  and  $\ov\tau(x):[0,T]\to[0,T]$ 
the 
first times at which $\ov x$
is attained: 
\[
\ov x_t:=\sup_{s\le t}x_s
\quad\text{and}\quad
\ov\tau_t(x):=\inf\{s\in[0,t]:\max\{x_s,x_{s-}\}=\ov x_t\},
\quad t\in[0,T],
\] 
where
$x_{t-}:=\lim_{s\uparrow t}x_s$ for all $t\in(0,T]$ and $x_{0-}:=x_0$.
The running infimum $\un x$ and its time of 
attainment process $\un\tau(x)$ are defined  analogously. 
The vectors 
$\ov\chi_t(x)=(x_t,\ov x_t,\ov\tau_t(x))$ and 
$\un\chi_t(x)=(x_t,\un x_t,\un\tau_t(x))$, $t\in[0,T]$, 
are referred to 
as \textit{extremal vectors} of $x$. 
%Let $\ti x:s\mapsto x_{t-s-}-x_t$ be the time reversal of $x$ on $[0,T]$. 
Note that $\un\chi_t(x)$ may be recovered from $\ov\chi_t(-x)=(-x_t,-\un x_t,\un\tau_t(x))$.
%and, if $\Delta x_t=0$,$\ov\chi^{\ti x}_t=(-x_t,-\un x,t-\un\tau^x)$,

\subsection{Projection of the convex minorant of a L\'evy process}

Let $X=(X_t)_{t\ge0}$ be a L\'evy process
started at zero (but not identically equal to zero) with \cadlag~paths.  Since $-X$ is a L\'evy process if and only if $X$ is, we may (and do) mostly focus our attention on one of the two extremal vectors. 
%For these reasons, we will focus on the study of %$\ov\chi^X$ going forward, 
%with the implicit understanding that analogous %results hold for $\un\chi^X$. 
%These triplets arise in all the problems discussed %at the beginning of the 
%book. 
At the core of all fluctuation theory results in this paper is the following theorem characterising explicitly the law of $\ov\chi_T(X)$ in terms of the increments of $X$ and an independent stick-breaking process $\ell$ on $[0,T]$
(recall that  $\ell=(\ell_n)_{n\in\N}$ and its remainder process $L=(L_n)_{n\in\N\cup\{0\}}$ are given by the recursion $L_0\coloneqq T$, $\ell_n\coloneqq V_nL_{n-1}$ and $L_n\coloneqq L_{n-1}-\ell_n$ for $n\in\N$, where   $(V_n)_{n\in\N}$ are iid~$\U(0,1)$). 

\begin{thm}[Stick-breaking representation of extrema]
\label{thm:SB-representation}
Let $X$ be any L\'evy process and 
$T>0$ a fixed time horizon. A stick-breaking process~$\ell$ on $[0,T]$ and its remainder $L$, independent of $X$, satisfy 
\begin{equation}
\label{eq:SB-representation}
\ov\chi_T(X)\eqd
\sum_{n=1}^\infty \big(X_{L_{n-1}}-X_{L_n},\max\{X_{L_{n-1}}-X_{L_n},0\},
    \ell_n\cdot\1_{\{X_{L_{n-1}}-X_{L_n}>0\}}\big).
%\qquad \xi_n:=X_{L_{n-1}}-X_{L_n},\qu%ad n\in\N.
\end{equation}
%where $L$ is the remainder process of %$\ell$. 
%Moreover, for any $n\in\N$, we have 
%\begin{equation}
%\label{eq:SB-finite-representation}
%\ov\chi_T(X)
%\eqd \ov\chi_{L_n}(X)+\sum_{k=1}^n %\big(\xi_k,\max\{\xi_k,0\},
%    \ell_k\cdot\1_{\{\xi_k>0\}}\big).	
%\end{equation}
\end{thm}

We refer to~\eqref{eq:SB-representation}
%--\eqref{eq:SB-finite-representation}
as the 
\emph{SB representation} of the extremal vector $\ov\chi_T(X)$. 
%Since for any $n\in\N$ the sequence $(\ell_k)_{k\geq %n}$ is a stick-breaking process on the interval %$[0,L_{n-1}]$, independent of $(\ell_k)_{1\le k\le %n-1}$, equality in law %in~\eqref{eq:SB-finite-representation} follows %from~\eqref{eq:SB-representation}. 
The SB representation in~\eqref{eq:SB-representation} is an easy direct consequence of our main result, Theorem~\ref{thm:levy-minorant} in Section~\ref{sec:convex-minorants} below, because 
$\ov\chi_T(X)$
equals the extremal vector of the concave majorant of $X$ on $[0,T]$.
%We will also drop the 
%explicit dependence of $\ov\tau$, $\un\tau$, %$\ov\chi$ and 
%$\un\chi$ on $X$ and the time horizon $T>0$ will %be specified 
%whenever its value is not implied by the context. 
%Let us point out that the two terms on the right %side of the quality %in~\eqref{eq:SB-finite-representation} are %conditionally 
%independent given $L_n$. Moreover, by the %definition of $L_n$, it is 
%a stochastic recursion with stationary marginal %law $(L_n,\ov\chi_{L_n}(Y))$. Put differently, it %represents the extremal 
%vector in terms of an extremal vector over an %exponentially small 
%interval (since $\e[L_n]=2^{-n}T$, see 
%Corollary~\ref{cor:SB-moments}) and conditionally %independent 
%increments of a L\'evy process over random %intervals. 
We stress that the power of the SB representation 
in Theorem~\ref{thm:SB-representation} for the extremal vector $\ov\chi_T(X)$
for any fixed time horizon $T$
lies in the fact that~\eqref{eq:SB-representation}
%--\eqref{eq:SB-finite-representation}
essentially reduce the properties of the path functional 
$\ov\chi_T(X)$
to the properties of the marginals of $X$.
We now illustrate this by deriving many of the classical highlights of the fluctuation theory of L\'evy processes from Theorem~\ref{thm:SB-representation}.
%Throughout the paper, the random %elements $Y$, $\ell$, $L$ %and 
%$\xi$ will be as in   %Theorem~\ref{thm:SB-representation}.
 %$\ell_n \eqd L_n$ for all $n\in\N$
%and
Note first that, since
$-\log(\ell_n/T)$ is gamma distributed with density $s\mapsto s^{n-1}\E^{-s}/(n-1)!$ for $s>0$,
for a measurable $f:[0,T]\to\R_+$ we have
\begin{equation}
\label{eq:stick_breaking_expectation}
%\e \sum_{n\in\N}f(L_n)= 
\e \sum_{n\in\N}f(\ell_n)=\int_0^T s^{-1}f(s) \D s.
\end{equation}

In the definition of $\ov\tau_t(X)$ (resp.  $\un\tau_t(X)$), we take the first rather than last
time the maximum (resp. minimum) is attained. Our first corollary shows that this choice makes little difference. 

\begin{cor}
	\label{cor:unique-maximum}
	A L\'evy process $X$ attains its maximum at a unique time a.s. if and only if $X$ is not a
	driftless compound Poisson process; then 
	$\ov\chi_t(X)\eqd (X_t,X_t-\un X_t,t-\un\tau_t(X))$ for all $t>0$.
\end{cor}

\begin{proof}
If $X$ is a driftless compound Poisson process, it has piecewise constant paths, making the time of the maximum not unique. Assume $X$ is not a driftless compound Poisson, then $\p(X_t=0)>0$ for at most countably many $t>0$. Indeed, either
the law of $X_t$ is diffuse or, by Doeblin's  lemma~\cite[Lem.~15.22]{MR1876169},
%if $X$ is not compound Poisson with %drift, then the law of $X_t$ is %diffuse by Doeblin's  %lemma~\cite[Lem.~15.22]{MR1876169}. 
$X$ is compound Poisson with drift $\mu\ne 0$. %then %, by~\cite[Thm~24.5]{MR3185174}, 
In the latter case, $\p(X_t=0)>0$ if and only if 
$-\mu t$
is in a countable set generated by the atoms of the L\'evy measure of $X$,
implying the claim.
%$-\mu t\in\bigcup_{n\in\N\cup\{0\}}F_%n$, where $F_0\coloneqq\{0\}$, $F_1$ %is the set of atoms of the L\'evy %measure $\nu$ of $X$ and %$F_{n+1}\coloneqq F_{n}+F_1$ for %$n\in\N$. Since the set %$\bigcup_{n\in\N\cup\{0\}}F_n$ is %countable, we have $\p(X_t=0)=0$ for %all but countably many $t>0$. 

The time-reversal process $X'=(X'_s)_{s\in[0,t]}$, defined as
$X'_s:=X_{(t-s)-}-X_t$, $s\in[0,t]$, has the same law as $(-X_s)_{s\in[0,t]}$, implying 
$\ov\tau_t(X')\eqd\un\tau_t(X)$.
The gap 
$(t-\ov\tau_t(X'))-\ov\tau_t(X)\geq 0$
between the time of the first and last maximum of $X$ 
has expectation equal to zero and is hence zero a.s. Indeed, since 
$\p(X_t= 0) > 0$
for at most countably many $t>0$,
 Theorem~\ref{thm:SB-representation} and~\eqref{eq:stick_breaking_expectation} yield 
\begin{equation*}
	t-\e[\un\tau_t(X)]-\e[\ov\tau_t(X)]
	=\e\sum_{n=1}^\infty \ell_n\1_{\{X_{L_{n-1}}=X_{L_n}\}}
	=\e\sum_{n=1}^\infty \ell_n\p(X_{\ell_n}=0|\ell_n)
	=\int_0^t \p(X_s=0)\D s= 0.
\end{equation*}
%implying 
%$\vartheta=0$ a.s. and 
The identity in law 
follows from $\ov\chi_t(X')\eqd\ov\chi_t(-X)$.
\end{proof}

\begin{cor}
	\label{cor:Spitzer}
	For any L\'evy process $X$, the following formulae hold for any 
	$t>0$:
	\begin{equation}
		\label{eq:Spitzer}
		\e[\ov\tau_t(X)] = \int_0^t \p(X_s>0)\D s
		\qquad\text{and}\qquad
		\e[\ov X_t] = \int_0^t (\e\max\{X_s,0\}/s)\D s.
	\end{equation}
\end{cor}

\begin{proof}
	Let $\rho(s):=\p(X_s>0)$ and take expectations in the third 
	coordinate of SB-representation~\eqref{eq:SB-representation}.  Fubini's theorem and the formula in~\eqref{eq:stick_breaking_expectation} imply
	\[
	\e\ov\tau_t(X)
	=\sum_{n=1}^\infty \e[\ell_n\rho(\ell_n)]
	=\int_0^t \rho(s)\D s\quad\text{for any $t>0$.}
	\]
 The proof of the formula for the supremum, based on~\eqref{eq:SB-representation} and~\eqref{eq:stick_breaking_expectation}, is analogous.
\end{proof}

\subsection{Wiener--Hopf factorisation and Rogozin's criteria}

Consider an exponential time horizon $T_\theta\sim\Exp(\theta)$ 
with parameter $\theta\in(0,\infty)$ (i.e. 
$\e T_\theta = 1/\theta$),
independent of the L\'evy process $X$.
%A version of 
%Theorem~\ref{thm:SB-representation} 
%on a random time interval %$[0,T_\theta]$ is as follows. 
Let $\ell^{(\theta)}=(\ell_n^{(\theta)})_{n\in\N}$ be a stick-breaking process with a random time horizon
$T_\theta$. 
The random measure  
$\sum_{n=1}^\infty \delta_{\ell_n^{(\theta)}}$
on $(0,\infty)$
is easily seen to be a Poisson point process (PPP)
(see Appendix~\ref{app:PPP_expon-time} below).
By~\eqref{eq:stick_breaking_expectation}
its mean measure satisfies
$\e \sum_{n\in\N} \delta_{\ell_n^{(\theta)}}(A) =\int_A t^{-1}\E^{-\theta t}\D t$
for a measurable $A$ ($\delta_z$ denotes the Dirac delta at the point $z$).
Let 
$F(t,\D x):=\p(X_t\in \D x)$ denote the law of
$X_t$ for any $t>0$.
Marking each point $\ell_n^{(\theta)}$
by a random real number sampled from the law
$F(\ell_n^{(\theta)},\cdot)$,
by the Marking Theorem~\cite[p.~55]{MR1207584},
produces
a PPP on
$(0,\infty)\times\R$.

\begin{prop}
\label{prop:SB-PPP}
Let the time horizon $T_\theta\sim\Exp(\theta)$ and the stick-breaking process $\ell^{(\theta)}$ be independent of the L\'evy process $X$. Define $\xi_n^{(\theta)}:=X_{L_{n-1}^{(\theta)}}-X_{L_{n}^{(\theta)}}$, where $L^{(\theta)}=(L_{k}^{(\theta)})_{k\in\N\cup\{0\}}$ is the remainder process associated to $\ell^{(\theta)}$. Then $\Xi_\theta:=\sum_{n=1}^\infty \delta_{(\ell_n^{(\theta)},\xi_n^{(\theta)})}$ is a Poisson point process with mean measure 
\begin{equation}
\label{eq:mu_theta}
\mu_\theta(\D t,\D x):=t^{-1}\E^{-\theta t}\p(X_t\in\D x)\D t,
\qquad
(t,x)\in(0,\infty)\times\R.
\end{equation}
\end{prop}

An immediate corollary of Theorem~\ref{thm:SB-representation} and Proposition~\ref{prop:SB-PPP}
characterises the laws of the supremum and infimum of $X$ on the exponential time horizon $T_\theta$.

\begin{cor}
\label{cor:Laplace-extremaLevy-Exp}
Let $T_\theta\sim\Exp(\theta)$ be independent of the L\'evy process $X$. Then the moment generating functions of $\ov X_{T_\theta}$ and $-\un X_{T_\theta}$ are given by the following formulae for any $u\ge 0$: 
\begin{align}
\label{eq:Laplace-supLevy-Exp}
\e\big[\E^{-u\ov X_{T_\theta}}\big]
&=\exp\bigg(\int_0^\infty \int_{(0,\infty)}
\big(\E^{-ux}-1\big)\E^{-\theta t}t^{-1}\p(X_t\in\D x)\D t\bigg),\\
\label{eq:Laplace-infLevy-Exp}
\e\big[\E^{u\un X_{T_\theta}}\big]
&=\exp\bigg(\int_0^\infty \int_{(-\infty,0)}
\big(\E^{ux}-1\big)\E^{-\theta t}t^{-1}\p(X_t\in\D x)\D t\bigg).
\end{align}
\end{cor}

\begin{proof}
By Theorem~\ref{thm:SB-representation}, $\ov X_{T_\theta}\eqd\int_{(0,\infty)^2}x\Xi_\theta(\D t,\D x)$, where $\Xi_\theta$ is  a PPP with mean measure $\mu_\theta$. Campbell's formula~\cite[p.~28]{MR1207584} implies~\eqref{eq:Laplace-supLevy-Exp}. Applying~\eqref{eq:Laplace-supLevy-Exp} to $-X$ yields~\eqref{eq:Laplace-infLevy-Exp}.
\end{proof}

%\section{Applications of the stick-breaking %representation}
%\label{sec:SB-representation-app}

%In this section we will discuss immediate %applications of the stick-breaking 
%representation of the extremal vector of %L\'evy processes. These applications are 
%focused on recovering classical results on %fluctuation theory and on the 
%regularity of the extremal vector. Other %implications in this direction will be 
%mentioned in %Section~\ref{sec:biblio-coupling}. %Applications on the topic of 
%the approximation of the extremal vector are %studied in detail in subsequent 
%chapters.

%\subsection{Rogozin, Sptizer and the %Wiener-Hopf factorisation}
%\label{subsec:fluctuation}
Recall that $0$ is \emph{regular} for the half-line $(0,\infty)$ if $X$ visits $(0,\infty)$ almost surely immediately after time 0, i.e. 
$\p(\bigcap_{t>0}\bigcup_{s\le t}\{X_s>0\})=1$. 
%Note that either $0$ is regular for $I$ and %then 
%$\p(\bigcup_{s\le t}\{X_s\in I\})=1$ for all %$t>0$ or 
%$\p(\bigcap_{t>0}\bigcup_{s\le t}\{X_s\in %I\})=0$ since the last event lies 
%in the $\p$-trivial $\sigma$-algebra %$\F_0^X$, by Blumenthal's 0-1 law.

\begin{thm}[Rogozin's criterion]
\label{thm:Rogozin-short}
The starting point $0$ of $X$ is regular for $(0,\infty)$ if and only if
\begin{equation}
\label{eq:Rogozin}
\int_0^1 t^{-1}\p(X_t>0)\D t=\infty.
\end{equation}
\end{thm}

\begin{proof}
%Fix any $t>0$. Note that $X_s>0$ for some $s\in(0,T]$ if and only if %$\ov X_t>0$. 
Let the time horizon $T_\theta\sim\Exp(\theta)$ and random sequences $\ell^{(\theta)}$ and $\xi^{(\theta)}$ be as in Proposition~\ref{prop:SB-PPP} above. As $t\mapsto\ov X_t$ is non-decreasing a.s., $0$ is not regular for $(0,\infty)$ 
%The event $\ov X_t=0$ has positive probability for some $t>0$ 
if and only if $\p(\ov X_{T_\theta}=0)>0$.
% with positive probability. The last event takes place iff the 
Since $\ov X_{T_\theta}\eqd\int_{A}x\Xi_\theta(\D t,\D x)$, 
where $A:=(0,\infty)\times (0,\infty)$,
the event $\{\ov X_{T_\theta}=0\}$ is equal to the event 
$\{\Xi_\theta(A)=0\}$
that the PPP $\Xi_\theta=\sum_{n=1}^\infty \delta_{(\ell_n^{(\theta)},\xi_n^{(\theta)})}$ has no points in $A$.
Thus, $0$ is not regular for $(0,\infty)$ 
if and only if 
\[
\p\left(\ov X_{T_\theta}=0\right)=\p\big(\Xi_\theta\left(A\right)=0\big)
=\exp\big(-\e\Xi_\theta\left(A\right)\big)
=\exp\left(-\int_0^\infty t^{-1}\E^{-\theta t}\p(X_t>0)\D t\right)>0
\]
for some positive $\theta$, which is equivalent to~\eqref{eq:Rogozin}. 
\end{proof}

We can now characterise the behaviour of $X$ as $t\to\infty$.

\begin{thm}[Rogozin]%Spitzer, 
\label{thm:Rogozin-long}
Possibly degenerate variables $\ov X_\infty\coloneqq\sup_{t\ge 0} X_t$ and $\un X_\infty\coloneqq\inf_{t\ge 0} X_t$ satisfy 
\begin{align}
\label{eq:Laplace-supX-infinity}
\e\big[\E^{-u\ov X_\infty}\big]
&=\exp\bigg(\int_0^\infty \int_{(0,\infty)}
\big(\E^{-ux}-1\big)t^{-1}\p(X_t\in\D x)\D t\bigg),\\
\label{eq:Laplace-infX-infinity}
\e\big[\E^{u\un X_\infty}\big]
&=\exp\bigg(\int_0^\infty \int_{(-\infty,0)}
\big(\E^{ux}-1\big)t^{-1}\p(X_t\in\D x)\D t\bigg),
\end{align}
for any $u\ge 0$. Define the integrals
\[
I_+\coloneqq\int_1^\infty t^{-1}\p(X_t>0)\D t
\qquad \& \qquad
I_-\coloneqq\int_1^\infty t^{-1}\p(X_t<0)\D t.
\]
Then the following statements hold for any non-constant L\'evy process $X$:
\item[~\nf{(a)}] if $I_+<\infty$, then $\ov X_\infty$ is non-degenerate ($\ov X_\infty <\infty$ a.s.) infinitely divisible and 
$\lim_{t\to\infty}X_t=-\infty$;
\item[~\nf{(b)}] if $I_-<\infty$, then $\un X_\infty$ is non-degenerate ($\un X_\infty >-\infty$ a.s.) infinitely divisible and 
$\lim_{t\to\infty}X_t=\infty$;
\item[~\nf{(c)}] if $I_+=I_-=\infty$, then 
$\limsup_{t\to\infty}X_t=-\liminf_{t\to\infty}X_t=\infty$. 
\end{thm}

\begin{proof}
Let $T_1\sim\Exp(1)$ be independent of $X$ and note 
$T_1/\theta \sim \Exp(\theta)$ for any $\theta>0$.
Since $\ov X_{T_1/\theta}\to \ov X_\infty$ as $\theta\to 0$ a.s.,
the corresponding Laplace transforms converge pointwise. Thus the monotone convergence theorem applied to the right-hand sides of~\eqref{eq:Laplace-supLevy-Exp}--\eqref{eq:Laplace-infLevy-Exp}
implies~\eqref{eq:Laplace-supX-infinity}--\eqref{eq:Laplace-infX-infinity}.
Identity~\eqref{eq:Laplace-supX-infinity} (resp.~\eqref{eq:Laplace-infX-infinity}) implies that 
$I_+<\infty$ (resp. $I_-<\infty$) if and only if 
$\e \exp(-u\ov X_{\infty})>0$ 
(resp. $\e \exp(u\un X_{\infty})>0$)
for all $u\geq0$. 
This implies part (c) and all but the limits in parts (a)~\&~(b).
%Recall from Corollary~\ref{cor:Laplace-extremaLevy-Exp} that the %Laplace transform of $\ov X_{T_1/\theta}$ is 
%\[
%\e\big[\E^{-u\ov X_{T_1/\theta}}\big]
%=\exp\left(\int_0^\infty \int_{(0,\infty)}
%\big(\E^{-ux}-1\big)\E^{-\theta t}t^{-1}\p(X_t\in\D x)\D t\right),
%\qquad u\ge 0.
%\]

%The limits in (a)~\&~(b) follow from identities~\eqref{eq:Laplace-supX-infinity}--\eqref{eq:Laplace-infX-infinity}, the strong Markov property, the Borel--Cantelli lemma and simple manipulations, see~\cite[p.~365]{MR3185174} for details. 
%Next note that if $I_-<\infty$ then $\e\big[\E^{u\un X_\infty}\big]>0$ for any $u>0$ and the transform converges to $1$ as $u\to 0$, implying that $\un X_\infty>-\infty$ a.s. On the other hand, the condition $I_+=\infty$ similarly gives $\ov X_\infty=\infty$ a.s. The process $X$ does not explode in finite time because it has \cadlag~paths, so these facts imply (a). The proof of (b) and (c) is similar.
It remains to prove the limit in (b), as the proof of the limit in (a) is analogous. First we show that $I_++I_-=\infty$. 
Since $X$ is not constant, by
$I_++I_-=\int_1^\infty t^{-1}(1-\p(X_t=0))\D t$,  it suffices to prove $\int_1^\infty t^{-1}\p(X_t=0)\D t<\infty$. 
As shown in the proof of Corollary~\ref{cor:unique-maximum}
above, if $X$ is not a driftless compound Poisson process, the function
$t\mapsto \p(X_t=0)$
is zero for Lebesgue a.e. $t>0$. 
%so the cases in Theorem~\ref{thm:Rogozin-long} above is %exhaustive. 
If $X$ is driftless compound Poisson, then the function $t\mapsto\sqrt{t}\p(X_t=0)$ is bounded. Indeed, suppose without loss of generality that $X$ has positive jumps and consider the a decomposition $X_t=Y_t+S_{N_t}$, where the compound Poisson process $Y$, %in the set $(-\infty,c)$, 
the random walk $S$ with strictly positive increments and the Poisson process $N$ with intensity $\lambda>0$ are independent. Then  
$\p(X_t=0)=\int_{\R}\p(S_{N_t}=-x)\p(Y_t\in\D x) \le\sup_{x\in\R}\p(S_{N_t}=-x)$.  For $x\in\R$ we have $\sum_{n=0}^\infty \1_{\{S_n=-x\}}\le 1$ 
(since $n\mapsto S_n$ is strictly increasing) and 
\begin{equation}
\label{eq:prob_zero_bound}
\p(S_{N_t}=-x)
\le \E^{-\lambda t}\sup_{m\in\N\cup\{0\}}\{(\lambda t)^m/m!\}%\sum_{n=0}^\infty \p(S_n=-x)=\E^{-\lambda t}\sup_{m\in\N\cup\{0\}} 
\e\sum_{n=0}^\infty \1_{\{S_n=-x\}}\le \E^{-\lambda t}\sup_{m\in\N\cup\{0\}}\{(\lambda t)^m/m!\},
\end{equation}
Since $m\mapsto \E^{-\lambda t}(\lambda t)^m/m!$ is maximised at an integer $m$ approximately equal to $\lambda t$, by Stirling's formula and~\eqref{eq:prob_zero_bound}, a multiple of $t^{-1/2}$ is an upper bound for $\p(X_t=0)$ for all sufficiently large $t$.

Assume $I_-<\infty$ and pick a sequence $a_n\uparrow\infty$ such that $\p(\un{X}_{\infty}<-a_n)\le 1/n^2$ for $n\in\N$. Since $I_-<\infty$, we must have $I_+=\infty$ and thus $\ov X_\infty=\infty$ a.s., implying that the hitting time $S_n$ of $X$ of the level $2a_n$ is a.s. finite for every $n\in\N$. Define the sets 
\[
B_n=\{X_t<a_n\text{ for some }t>S_n\}\subset\{X_t-X_{S_n}<-a_n\text{ for some }t>S_n\},
\]
and note that $\p(B_n)\le\p(\un{X}_{\infty}<-a_n)\le 1/n^2$ by the strong Markov property. The Borel--Cantelli lemma then shows that $\p(B_n\text{ i.o.})=0$ and hence $\p(\liminf_{t\to\infty}X_t<\infty)=0$. In other words, we have $\lim_{t\to\infty}X_t=\infty$ a.s., completing the proof.
\end{proof}

Another easy corollary of 
Theorem~\ref{thm:SB-representation}
is the Wiener-Hopf factorisation. 

\begin{thm}[Wiener-Hopf factorisation]
\label{thm:WH}
Let the time horizon $T_\theta\sim\Exp(\theta)$ be independent of $X$.
%\item[\nf{(a)}] 
The random vectors $(\ov\tau_{T_\theta}(X),\ov X_{T_\theta})$ and $(T_\theta-\ov\tau_{T_\theta}(X), X_{T_\theta}-\ov X_{T_\theta})$ are independent, infinitely divisible with respective Fourier-Laplace transforms  given by 
\begin{align}
	\label{eq:WH+}
	\Psi_\theta^+(u,v)
	&:=\e\big[\E^{u \ov\tau_{T_\theta}(X)+v\ov X_{T_\theta}}\big]
	=\frac{\varphi_+(-\theta,0)}{\varphi_+(u-\theta,v)},\\
	\label{eq:WH-}
	\Psi_\theta^-(u,-v)
	&:=\e\big[\E^{u (T_\theta-\ov\tau_{T_\theta}(X))-v( X_{T_\theta}-\ov X_{T_\theta})}\big]
	=\frac{\varphi_-(-\theta,0)}{\varphi_-(u-\theta,v)},
\end{align}
for any $u,v\in\C$ with $\Re u,\Re v\le 0$.
Here
$\varphi_\pm$ is defined as follows:
set $A_+:=(0,\infty)$, $A_-:=(-\infty,0]$, 
\begin{equation}
	\label{eq:WH-FL}
	\varphi_\pm(a,b)
	:=\exp\left(\int_0^\infty\int_{A_\pm}
	\big(\E^{-t}-\E^{at+b|x|}\big)t^{-1}\p(X_t\in\D x)\D t\right),
\end{equation}
for any $a,b\in\C$ such that the integrals in~\eqref{eq:WH-FL} exist, including the  $\Re a<0$, $\Re b\le 0$.
%\item[\nf{(b)}] 
The characteristic exponent $\Psi$ of $X_1$ (i.e. $\e \exp(vX_1)=\exp\Psi(v)$ for $v\in\C$ with $\Re v =0$)
satisfies
\begin{equation}
    \label{eq:WH-product}
%\frac{\theta}{\theta-u-\Psi(v)}
\theta/(\theta-u-\Psi(v))
=\Psi_\theta^+(u,v)\Psi_\theta^-(u,v),
\qquad u,v\in \C\text{ with $\Re v=\Re u=0$.}
\end{equation}
%\item[\nf{(c)}] The following identity holds: 
%$\Psi(v) =-\lim_{u\uparrow0}\varphi_+(u,v)\varphi_-(u,-v)$
%for
%$v\in \C$ with $\Re v=0$.
\end{thm}

\begin{proof}
%(a) 
Let $\ell^{(\theta)}$, $\xi^{(\theta)}$ 
and $\Xi_\theta=\sum_{n=1}^\infty \delta_{(\ell_n^{(\theta)},\xi_n^{(\theta)})}$
be as in 
Proposition~\ref{prop:SB-PPP}.
By Theorem~\ref{thm:SB-representation}, we have 
\begin{align*}
	(\ov\tau_{T_\theta}(X),\ov X_{T_\theta}) \eqd \int_{B_+}(t,x)\Xi_\theta(\D t,\D x)\quad \& \quad
	(T_\theta-\ov\tau_{T_\theta}(X),X_{T_\theta}-\ov X_{T_\theta}) \eqd \int_{B_-}(t,x)\Xi_\theta(\D t,\D x),
\end{align*}
where $B_\pm:=(0,\infty)\times A_\pm$.
Moreover, since the joint law of
$(\ov\tau_{T_\theta}(X),\ov X_{T_\theta})$ and $(T_\theta-\ov\tau_{T_\theta}(X), X_{T_\theta}-\ov X_{T_\theta})$
equals that of the two integrals in the display above, the vectors are independent because $B_+\cap B_-=\emptyset$. 
By Proposition~\ref{prop:SB-PPP},
the mean measure of
$\Xi_\theta$
equals
$\mu_\theta(\D t,\D x)=t^{-1}\E^{-\theta t}\p(X_t\in\D x)\D t$.
 %:=\sum_{n=1}^\infty %\delta_{(\ell_n^{(\theta)},\xi_n^{(\theta)})}$ is a Poisson point process %with mean measure 
Hence
\begin{align*}
\Psi^\pm_\theta(u,v)
%&=\e\big[\E^{u\ov\tau_{T_\theta}+v\ov X_{T_\theta}}\big]
&=\exp\bigg(\int_{(0,\infty)\times A_\pm}\big(\E^{ut+vx}-1\big)\frac{\E^{-\theta t}}{t}\p(X_t\in \D x)\D t\bigg)\quad\text{for all $u,v\in\C$ with $\Re u,\Re v\le 0$,}
%\label{eq:Psi+}
%=\exp\left(\int_{(0,\infty)\times A_+}\big(\E^{ut+vx}-1\big)
%& = \frac{\E^{-\theta t}}{t}\p(X_t\in \D x)\D t\right),\\
%\\
%\label{eq:Psi-}
%\Psi^-_\theta(u,v)
%&=\exp\left(\int_{(0,\infty)\times A_-}\big(\E^{ut+vx}-1\big)
%\frac{\E^{-\theta t}}{t}\p(X_t\in \D x)\D t\right),
\end{align*}
by Campbell's Theorem~\cite[p.~28]{MR1207584}.
This representation of $\Psi^\pm_\theta(u,v)$
and~\eqref{eq:WH-FL}
imply~\eqref{eq:WH+}--\eqref{eq:WH-}. 
The independence and the formula $\e \exp(u T_\theta+v X_{T_\theta}) = \theta/(\theta-u-\Psi(v))$
imply identity~\eqref{eq:WH-product}. 
%A similar relation holds for the pair %$(T_\theta-\ov\tau_{T_\theta},\ov X_{T_\theta}-X_{T_\theta})$. These %pairs have i.d. distributions because they are the integrals of %Poisson point processes. The formulae %in~\eqref{eq:WH+},~\eqref{eq:WH-} and~\eqref{eq:WH-FL} follow from %simple algebraic manipulations, completing the proof of part (a).
\begin{comment}
(b) For $u,v\in \C$ with $\Re v=\Re u=0$, the previous display and the Frullani integral imply
\begin{align*}
	\Psi_\theta^+(u,v)\Psi_\theta^-(u,v)
	&=\exp\left(\int_0^\infty \int_{\R} \big(\E^{ut+vx}-1\big)
	\frac{\E^{-\theta t}}{t}\p(X_t\in \D x)\D t\right)\\
	&=\exp\left(\int_0^\infty \big(\E^{(u+\Psi(v))t}-1\big)
	\frac{\E^{-\theta t}}{t}\D t\right)
	=\frac{\theta}{\theta-u-\Psi(v)}.
\end{align*}
(c) For $\theta>0$, we have  
$\varphi_+(-\theta,0)\varphi_-(-\theta,0)
	=\exp\left(\int_0^\infty \big(\E^{-t}-\E^{-\theta t}\big)t^{-1}\D t\right)$. Hence  we find
\begin{align*}
	\varphi_+(-\theta,0)\varphi_-(-\theta,0)
%	=\exp\left(\int_0^\infty \big(\E^{-t}-\E^{-\theta t}\big)t^{-1}\D %t\right)
	=\exp\left(\int_0^\infty \big(1-\E^{-\theta t}\big)\frac{\E^{-t}}{t}\D t
	-\int_0^\infty \big(1-\E^{-t}\big)\frac{\E^{-\theta t}}{t}\D t\right)
	= \theta,
\end{align*}
using the Frullani integral. Part (b) then yields 
\[
z-\Psi(v)
=\varphi_+(-z,v)\varphi_-(-z,-v),
\quad\text{where}\quad
z=\theta-u.
\]
Taking $z\to 0$ yields part (c) and concludes the proof.
\end{comment}
\end{proof}

%study the existence and positivity of the density of %$\ov\chi_t(X)$.
%could be used to derive a version of Theorem~\ref{thm:abs_cont}. %In fact, the authors proved a result similar to %Theorem~\ref{thm:abs_cont} using the same strategy we explained %and~\cite[Thm~1]{MR2978134}. Note, however, that the proof %of~\cite[Thm~1]{MR2978134} requires the following result.

%\begin{thm}[Rogozin]
%\label{thm:IV_at_0}
%If a L\'evy process $X$ has paths of infinite variation, 
%(i.e. its L\'evy measure $\nu$ satisfies $\int_{(-1,1)}|x|\nu(\D %x)=\infty$) 
%then $$\limsup_{t\da0} X_t/t=-\liminf_{t\da 0}X_t/t=\infty\quad\text{a.s.}$$
%\end{thm}

A circular argument would arise if one attempted to develop fluctuation theory for L\'evy processes with diffuse transition laws using~\cite[Thm~1]{MR2978134}, because of its reliance on Rogozin's result
for L\'evy process of infinite variation,
%(i.e. its L\'evy measure $\nu$ satisfies $\int_{(-1,1)}|x|\nu(\D %x)=\infty$) 
$\limsup_{t\da0} X_t/t=-\liminf_{t\da 0}X_t/t=\infty$ a.s.,
which relies on the Wiener-Hopf factorisation in an essential way. 
In contrast, Theorem~\ref{thm:levy-minorant},
applicable to all L\'evy processes, has an elementary proof
that does not use fluctuation theory.
In fact, Theorem~\ref{thm:levy-minorant} can be used as a short-cut to Rogozin's result as it implies the necessary fluctuation theory as described in this paper.

The question of the absolute continuity of the law of $\ov\chi_t(X)=(X_t, \ov X_t, \ov \tau_t(X))$ was the main topic in~\cite{MR3098676}, investigated using excursion theory. Again, Theorem~\ref{thm:SB-representation} provides an easy approach. In fact, in~\cite[Thm~2]{MR2978134} the authors use a version of Theorem~\ref{thm:SB-representation} for diffuse L\'evy processes (their~\cite[Thm~1]{MR2978134}) to show that the laws of $(\ov\tau_t(X),\ov X_t)$ and $(\ov X_t-X_t,\ov X_t)$ are equivalent to Lebesgue measure on $(0,t]\times(0,\infty)$ and $(0,\infty)^2$, respectively, for any L\'evy process with absolutely continuous marginals. 

%as a starting point to prove the fluctuation identities presented %above is that~\cite[Thm~1]{MR2978134} does not hold for %non-diffuse L\'evy processes and its proof requires %Theorem~\ref{thm:IV_at_0}. A proof some of the fluctuation %identities above using~\cite[Thm~1]{MR2978134} would not %circumvent the use of potential and excursion theory since the %proof of Theorem~\ref{thm:IV_at_0} requires some of the %fluctuation identities above, elementary properties of L\'evy %processes and the Borel-Cantelli lemmas, %see~\cite[p.~351]{MR3185174}. The benefit of our approach is that, %in the proof of Theorem~\ref{thm:SB-representation}, we only %require elementary properties of L\'evy processes, simple %invariance principles for random walks and an analysis of the %5convergence of piecewise linear convex functions.

\subsection{Vertex process of a L\'evy process}
\label{subsec:vertex}

The  final application of Theorem~\ref{thm:levy-minorant} describes the law of the vertex process of the convex minorant of $X$. Intuitively, the vertex process is naturally paremetrised by the slope of the  minorant and its range coinciding with the extremal points of the graph of the convex minorant. 
In the infinite time horizon case, Groeneboom~\cite{MR714964} described the law of the vertex process of Brownian motion as a time inhomogeneous additive process (i.e. a process with independent but non-stationary increments). This description was later extended by Nagasawa~\cite[Ch.~XI.1]{MR1739699} to L\'evy processes with infinite activity, again over an infinite time horizon.  We now show, as a simple corollary of Theorem~\ref{thm:levy-minorant}, that the vertex process has independent increments for all L\'evy processes and independent exponential (possibly infinite) time horizons.  
Before defining the vertex process, note that the convex minorant $\CM{\infty}{X}$ of $X$ on the infinite time interval $[0,\infty)$ exists and is finite if and only if $l:=\liminf_{t\to\infty}X_t/t$, which is a.s.~constant, lies in $(-\infty,\infty]$. Otherwise, we have $\CM{\infty}{X}\equiv-\infty$ on $(0,\infty)$ (see Section~\ref{sec:convex-minorants} below for a characterisation of $l=-\infty$ in terms of the L\'evy measure of $X$). 

For $\theta\in[0,\infty)$,
let the exponential random variable $T_\theta\sim\Exp(\theta)$ (with $T_0:=\infty$ a.s.) be independent of $X$. The right-derivative of $\CM{T_\theta}{X}$, given by 
$(\CM{T_\theta}{X})'(t):=\lim_{h\downarrow0}(\CM{T_\theta}{X}(t+h)-\CM{T_\theta}{X}(t))/h$, exists for any $t\in[0,T_\theta)$ and is a non-decreasing function of $t$ with a possibly infinite limit at $T_\theta$. 
Define 
the \textit{vertex process} 
$(\sigma,\eta)$
of $X$ to be 
the \cadlag~process 
 parametrised by slope $s\in\R$, where 
$\sigma=(\sigma_s)_{s\in\R}$ 
is the right-inverse of $(\CM{T_\theta}{X})'$ and $\eta=(\eta_s)_{s\in\R}$  the value of the minorant at that time, 
\begin{equation}
\label{eq:vertex_process_def}
\sigma_s:=T_\theta\wedge \inf\big\{t\in[0,T_\theta)
    \,:\,(\CM{T_\theta}{X})'(t)>s\big\}
\quad\text{and}\quad \eta_s:=\CM{T_\theta}{X}(\sigma_s)
=\min\{X_{\sigma_s-},X_{\sigma_s}\},
\end{equation}
respectively (here $a\wedge b:=\min\{a,b\}$ and $\inf\emptyset :=\infty$).
  Its construction is somewhat reminiscent of a ladder subordinator, indexed by the local time of $X$ at its running minimum, appearing in the classical approach to fluctuation theory of L\'evy processes, see e.g.~\cite[Ch.~VI]{MR1406564}.  It is clear that $\sigma$ has non-decreasing paths and, almost surely, we have 
\begin{equation}
    \label{eq:basic_properties_vertex}
\lim_{s\to-\infty}\sigma_s=0,\quad \lim_{s\to-\infty}\eta_s=X_0=0,
\quad
\eta_0=\un X_{T_\theta},
\quad
\lim_{s\to\infty}\sigma_s=T_\theta,
%\quad\text{and}
\quad
\lim_{s\to\infty}\eta_s=\CM{T_\theta}{X}(T_\theta),
\end{equation}
where  
$\CM{T_\theta}{X}(T_\theta)$
%=\lim_{t\to T_\theta}\CM{T_\theta}{X}(t)$
%exists a.s., 
equals 
$X_{T_\theta}$ (if $\theta>0$) or
$\sgn(l)\cdot\infty$ (if $\theta=0$), where 
$\sgn(l)=1$ 
if 
$l=\liminf_{t\to\infty}X_t/t>0$
and $-1$ otherwise. 
The following result is an elementary consequence of the Poissonian structure of $\CM{T_\theta}{X}$, described in Corollary~\ref{cor:SB-PPP} below, which extends Proposition~\ref{prop:SB-PPP} above.

\begin{thm}\label{thm:additive}
Pick $\theta\in [0,\infty)$. Extend the definition of the mean measure $\mu_\theta$   in~\eqref{eq:mu_theta} (for $\theta>0$) by $\mu_0(\D t,\D x):=\1_{\{x/t<l\}}t^{-1}\p(X_t\in\D x)\D t$ on $(t,x)\in(0,\infty)\times\R$. Then the vertex process $(\sigma,\eta)$ has independent increments and its Fourier-Laplace transform is given by 
\begin{equation}
\label{eq:cf_inc_mu}
\e\big[\E^{u \sigma_s + v \eta_s}\big]
=\exp\bigg( \int_{(0,\infty)\times\R} 
    (\E^{ut + vx}-1)\1_{\{x/t \leq s\}}
        \mu_\theta(\D t,\D x) \bigg),
\end{equation}
for any $s\in\R$, $u,v\in\C$ with $\Re u\le 0$ and $\Re v=0$. In particular, the Laplace transform of $\sigma_s$ equals  
\begin{equation*}
%\label{eq:cf_sigma}
\e[\E^{-u \sigma_s}]
=\exp\bigg(\int_0^\infty (e^{-u t}-1)e^{-\theta t}\p(X_t\le st)\frac{\D t}{t}\bigg),
%\quad\text{$u\ge0$, $s\in\R$ (if $\theta>0$) and $s<l$ (if $\theta=0$)}.
\end{equation*}
for all $u\ge0$, and either $s\in\R$ (if $\theta>0$) or $s\in(-\infty,l)$ (if $\theta=0$).
\end{thm}

\begin{proof}
Let $\Xi_\theta\coloneqq\sum_{n=1}^\infty\delta_{(\ell_n^{(\theta)},\xi_n^{(\theta)})}$ be a PPP on $A:=(0,\infty)\times\R$ with mean measure~$\mu_\theta$. By Corollary~\ref{cor:SB-PPP} below and definition~\eqref{eq:vertex_process_def} of the vertex process $(\sigma,\eta)$, for any sequence of slopes $s_1<\ldots<s_n$ in $\R$ we have 
\[
\big((\sigma_{s_1},\eta_{s_1}),\ldots,(\sigma_{s_n},\eta_{s_n})\big)
\eqd \int_A
    \big(\1_{\{x/t\le s_1\}}(t,x),\ldots,\1_{\{x/t\le s_n\}}(t,x)\big)\Xi_\theta(\D t,\D x).
\]
Intuitively, since the convex minorant $\CM{T_\theta}{X}$ is piecewise linear, $\sigma_{s_1}$ is the sum of all the horizontal lengths of all the linear faces with slope strictly smaller than $s_1$, see Section~\ref{subsec:CM-PL} for a precise description. 
The increments $(\sigma_{s_1},\eta_{s_1})$, $(\sigma_{s_2}-\sigma_{s_1},\eta_{s_2}-\eta_{s_1})$, \ldots, $(\sigma_{s_n}-\sigma_{s_{n-1}},\eta_{s_n}-\eta_{s_{n-1}})$ are equal in law to the integrals with respect to the PPP $\Xi_\theta$ over disjoint ``pizza slices'' $\{(t,x)\in A: x/t\le s_1\}$, $\{(t,x)\in A: s_1< x/t\le s_2\}$, \dots, $\{(t,x)\in A: s_{n-1}<x/t\le s_n\}$ and are thus independent. The  Fourier-Laplace transform of the marginal $(\sigma_s,\eta_s)$ follows from Campbell's formula~\cite[p.~28]{MR1207584}.
\end{proof}

The vertex process can be constructed from the path of $X$ without a reference to the convex minorant $\CM{T_\theta}{X}$ as follows. Given $s\in\R$ (with $s<l$ if $\theta=0$), define the L\'evy process $X^{(s)}=(X^{(s)}_t)_{t\geq0}$ by $X^{(s)}_t\coloneqq X_t - st$. Then 
$\sigma_s=\un\tau_{T_\theta}(X^{(s)})$ and $\eta_s-s\sigma_s=\un{X}_{T_\theta}^{(s)}$. This description and Theorem~\ref{thm:additive} yield the following novel generalisation of the classical Wiener-Hopf factorisation.

\begin{cor}
\label{cor:multiple_drifts}
Pick $\theta> 0$ and let $T_\theta\sim\Exp(\theta)$ be independent of $X$.  Then, for any real numbers $s_1<\cdots<s_n$, the following $n+1$ vectors are independent: 
\begin{gather*}
(\un\tau_{T_\theta}(X^{(s_1)}),\un{X}_{T_\theta}^{(s_1)}
        +s_{1}\un\tau_{T_\theta}(X^{(s_1)})),
\quad
(T_\theta-\un\tau_{T_\theta}(X^{(s_{n})}),X_{T_\theta}-\un{X}_{T_\theta}^{(s_n)}
        -s_n\un\tau_{T_\theta}(X^{(s_n)})),
\quad\text{and}\\
(\un\tau_{T_\theta}(X^{(s_{i+1})})-\un\tau_{T_\theta}(X^{(s_{i})}),
    \un{X}_{T_\theta}^{(s_{i+1})}
        -\un{X}_{T_\theta}^{(s_{i})}
    +s_{i+1}\un\tau_{T_\theta}(X^{(s_{i+1})})
        -s_{i}\un\tau_{T_\theta}(X^{(s_{i})})),
\quad i=1,\ldots,n-1.
\end{gather*}
\end{cor}

\section{Stick-breaking representations for convex minorants}
\label{sec:convex-minorants}

\subsection{The convex minorant of a L\'evy process (the main theorem)}
\label{subsec:main_theorem}
Given any \cadlag~function $x:[0,T]\to\R$, its convex minorant, denoted by $\CM{T}{x}$, is the largest convex function that is pointwise smaller than $x$. The goal of this section is to prove our main result, Theorem~\ref{thm:levy-minorant}, which (when applied to $-X$) clearly yields Theorem~\ref{thm:SB-representation} above.

\begin{thm}\label{thm:levy-minorant}
Let $X$ be a L\'evy process and fix $T>0$. Let $(\ell_n)_{n\in\N}$ be a uniform stick-breaking process on $[0,T]$ independent of $X$. Then the convex minorant $C_T^X$ of $X$ has the same law (in the space of continuous functions on $[0,T]$) as the piecewise linear convex function on $[0,T]$ given by the formula
\begin{equation}
\label{eq:levy-minorant}
\begin{gathered}
t\mapsto\sum_{n=1}^\infty \xi_n\min\{\max\{t-a_n,0\}/\ell_n,1\},
%\quad t\in[0,T],
\quad\text{where}\quad
\xi_n\coloneqq X_{L_{n-1}}-X_{L_n} \text{ and}\\
a_n\coloneqq \sum_{k=1}^\infty \ell_k \cdot\1_{\{\xi_k/\ell_k<\xi_n/\ell_n\}} 
+ \sum_{k=1}^{n-1} \ell_k \cdot\1_{\{\xi_k/\ell_k=\xi_n/\ell_n\}},
%\quad\text{and}\quad
%\xi_n\coloneqq X_{TL_{n-1}}-X_{TL_n},
\quad n\in\N.
\end{gathered}
\end{equation}
In particular, the face of the piecewise linear function with horizontal length~$\ell_n$ has vertical height~$\xi_n$. \end{thm}

We now present a simple but useful consequence of  Theorem~\ref{thm:levy-minorant}, first established in~\cite[Cor.~2 \& 3]{MR2978134} for diffuse L\'evy processes. Recall the a.s. constant  $l=\liminf_{t\to\infty}X_t/t$ lies in $[-\infty,\infty]$. Whenever the expectation $\e X_1$ is well defined, i.e., if $\min\{\e\max\{X_1,0\},\e\max\{-X_1,0\}\}<\infty$, the strong law of large numbers~\cite[Thms~36.4 \& 36.5]{MR3185174} implies that $l=\e X_1=\lim_{t\to\infty}X_t/t$ a.s. Otherwise, we have $\e X_1^+=\e X_1^-=\infty$ and~\cite[Thm~15]{MR2320889} implies that $l\in \{-\infty,\infty\}$ and 
\begin{equation}
\label{eq:l_is_minus_infinity}
l=-\infty \qquad \text{if and only if} \qquad 
\int_{(-\infty,-1)}
    \frac{|x|}{1+\int_{1}^{|x|}\nu([y,\infty))\D y}\nu(\D x)=\infty.
\end{equation}
The proof of~\cite[Thm~15]{MR2320889} is an easy corollary of the analogous result for random walks (see also~\cite{MR336806}), proven using renewal theory. Indeed, the small-jump and Brownian components of $X$ converge by the strong law of large numbers and the big-jump component is a  random walk, time-changed by a Poisson process.

\begin{cor}
\label{cor:SB-PPP}
Let $\theta\in[0,\infty)$ and $T_\theta\sim\Exp(\theta)$ (with $T_0=\infty$) be independent of $X$. When $\theta=0$ we assume that $l=\liminf_{t\to\infty} X_t/t>-\infty$. Define a $\sigma$-finite measure on $(0,\infty)\times\R$:
\[
\mu_\theta(\D t,\D x):
=\begin{cases}
t^{-1}\E^{-\theta t}\p(X_t\in\D x)\D t,
& \theta>0,\\
\1_{\{x/t<l\}}t^{-1}\p(X_t\in\D x)\D t,
& \theta=0.
\end{cases}
\]
Let $\Xi_\theta=\sum_{n\in\N}\delta_{(\ell_n^{(\theta)},\xi_n^{(\theta)})}$ be a Poisson point process with mean measure $\mu_\theta$. Then $\CM{T_\theta}{X}$ has the same law as the piecewise linear function given by~\eqref{eq:levy-minorant}. In particular, the face of the piecewise linear function with horizontal length $\ell_n^{(\theta)}$ has vertical height $\xi_n^{(\theta)}$ and,  when $\theta=0$, the corresponding slope $\xi_n^{(\theta)}/\ell_n^{(\theta)}$ lies on the interval $(-\infty,l)$. 
\end{cor}

\begin{proof}
The result for $\theta>0$ follows from Proposition~\ref{prop:SB-PPP} and Theorem~\ref{thm:levy-minorant}. We now describe the case $\theta=0$ following the proof of~\cite[Cor.~3]{MR2978134}. First note that the definition of $l$ implies that the right-derivative of $\CM{\infty}{X}$ is upper bounded by $l$. On the other hand, the derivative has no smaller upper bound. Indeed, suppose the derivative is bounded by some $a<l$. Then there exists some time $T>0$ after which $X_t/t\ge (a+l)/2$ for all $t\ge T$. Thus  $t\mapsto \CM{\infty}{X}(\min\{t,T\})+\max\{t-T,0\}(a+l)/2$ is convex, dominates $\CM{\infty}{X}$ and is dominated by $X$, contradicting the maximality of $\CM{\infty}{X}$. Moreover, the right-derivative of $\CM{\infty}{X}$ is never equal to $l$. If $l=\infty$ this is clear. If $l<\infty$, by the discussion preceding the corollary we have $\e|X_1|<\infty$ and thus $\e X_1 =l$. Since the martingale $ (X_t-lt)_{t\geq0}$ is recurrent~\cite[Thm~36.7]{MR3185174}, then $\liminf_{t\to\infty}\CM{\infty}{X}(t)-lt\le \liminf_{t\to\infty}X_t-lt=-\infty$. Thus, the right-derivative of $\CM{\infty}{X}$ is strictly smaller than $l$ on $[0,\infty)$ and its limit as $t\to\infty$ equals $l$. 

For any $a<l$ let $S_a$ be the first time the derivative of $\CM{\infty}{X}$ is greater than $a$. Note that for any $T>S_a$, the convex minorant $\CM{T}{X}$ agrees with $\CM{\infty}{X}$ on $[0,S_a]$. Since $S_a\to\infty$ as $a\to l$, this establishes the piecewise linearity of $\CM{\infty}{X}$. Let $E$ be an independent exponential time with unit mean and let $T_\lambda:=E/\lambda$ for $\lambda>0$. Note that $\CM{T_\lambda}{X}$ converges to $\CM{\infty}{X}$ as $\lambda\to 0$ since, for every $a<l$, both convex minorants agree on $[0,S_a]$ for all sufficiently small $\lambda$. Thus, for every $a<l$, the associated PPP encoding the lengths and heights of the faces of $\CM{T_\lambda}{X}$ (as in Proposition~\ref{prop:SB-PPP} and Theorem~\ref{thm:levy-minorant}) converges to the corresponding point process for the faces of $\CM{\infty}{X}$ on the set $Z_a:=\{(t,x)\in[0,\infty)\times\R\,:\,x/t<a\}$ as $\lambda\to0$, implying that the PPP representation of the faces of $\CM{\infty}{X}$ mean measure $\mu_0$ holds on $Z_a$. Since $a<l$ is arbitrary, $S_a\to\infty$ as $a\to l$ and $Z_l=\bigcup_{a<l}Z_a$, the PPP representation extends to $[0,\infty)$, completing the proof. 
\end{proof}

\noindent\underline{Overview of the proof of Theorem~\ref{thm:levy-minorant}.} 
Theorem~\ref{thm:levy-minorant} connects two worlds: (\textbf{W1}) $X$ and its convex minorant $C_T^X$ on the interval $[0,T]$ and (\textbf{W2}) a random piecewise linear convex function. %given by the formula in~\eqref{eq:levy-minorant}.  
We first establish a convergence result within (\textbf{W2}) for a sequence of piecewise linear convex functions, see Section~\ref{subsec:CM-PL}. This crucial step in the proof requires only elementary geometric manipulations of piecewise linear convex functions. In Section~\ref{subsec:CM-RW}, using~\cite[Thm~1]{MR2825583}, we establish a bridge between (\textbf{W1}) and (\textbf{W2}) for random walks. 
We recall the 3214 path transformation~\cite{MR2825583} for random walks and provide a short proof, based on the convergence results in Section~\ref{subsec:CM-PL}, of the connection between (\textbf{W1}) and (\textbf{W2}) for random walks with general increments, see Theorem~\ref{thm:SB-representation-RW} below. In Section~\ref{subsec:proof-main-thm}, we establish Theorem~\ref{thm:levy-minorant}
by taking the limit of the convex minorant of the random-walk skeleton of $X$ 
in (\textbf{W1}) and, using the convergence results of Section~\ref{subsec:CM-PL}, the corresponding limit 
in (\textbf{W2}).

We stress that the proof of Theorem~\ref{thm:levy-minorant}
given in this section is self-contained, requiring only rudimentary real analysis and the fact that $X$ has stationary, independent increments and right-continuous paths with left limits. In particular, we make no use of the L\'evy measure, the L\'evy-Khintchine formula for $X$ or weak convergence in the $J_1$-topology
on the Skorokhod space. 

%The road to the proof of %Theorem~\ref{thm:levy-minorant} requires some %continuity properties of convex minorants, %piecewise linear functions, invariance results %for random walks and approximations for L\'evy %processes. The proof of the results in the next %subsection may appear complicated at first; %however, drawing pictures reveals that the %arguments are in fact simple, but the precise %description is lengthy.

\subsection{Convex minorants and piecewise linear functions}
\label{subsec:CM-PL}

Let $\br{n}\coloneqq\{1,\ldots,n\}$ for $n\in\N$ and $\br{\infty}\coloneqq\N$. We say that a function $f:[0,T]\to\R$ is piecewise linear if there exists a set consisting of $N\in\ov\N\coloneqq\N\cup\{\infty\}$ pairwise disjoint non-degenerate subintervals $\{(a_n,b_n):n\in \br{N}\}$ of $[0,T]$ such that $\sum_{n=1}^N( b_n-a_n)=T$ and $f$ is linear on each $(a_n,b_n)$. A face of $f$, corresponding to a subinterval $(a_n,b_n)$, has length $l_n=b_n-a_n>0$, height $h_n=f(b_n)-f(a_n)\in\R$ and slope $h_n/l_n$. Note that, if $f$ is continuous and of finite variation  $\sum_{n=1}^{N} |f(b_n)-f(a_n)|<\infty$, the following representation holds: 
\begin{equation}
\label{eq:pw_linear_rep}
f(t)=f(0)+\sum_{n=1}^N h_n\min\{\max\{t-a_n,0\}/l_n,1\}, \qquad t\in[0,T].
\end{equation}

The number $N$ in representation~\eqref{eq:pw_linear_rep} is not unique in general as any face may be subdivided into two faces with the same slope. Moreover, for a fixed $f$ and $N$, the set of intervals $\{(a_n,b_n):n\in\br{N}\}$ need not be unique. Furthermore we stress that the sequence of faces in~\eqref{eq:pw_linear_rep} does not necessarily respect the chronological order. Put differently, the sequence $(a_n)_{n\in\br{N}}$ need not be increasing. %In fact, we say the interval $(a,b)$ of a face is \emph{maximal} if $f$ is not linear on any $(a',b')\supsetneq(a,b)$. 
%Lemmas from $\ve$SS 
We use the convention $\sum_{k=n}^m=0$ when $n>m$ 
and denote $x^+:=\max\{x,0\}$ for all $x\in\R$, throughout.

\begin{lem}
\label{lem:seq-to-cm}
Fix $T>0$, $N\in\ov\N$ and let $l=(l_n)_{n=1}^N$ be a sequence of positive lengths with $\sum_{n=1}^N l_n=T$.\\ 
{\nf{(a)}} For any sequence of heights $h=(h_n)_{n=1}^N$ with $\sum_{n=1}^N |h_n|<\infty$, the function 
\begin{equation}\begin{split}
\label{eq:CM-formula}
F_{l,h}(t)&\coloneqq \sum_{n=1}^N h_n\min\{(t-a_n)^+/l_n,1\},
%\hspace{8mm}
\qquad
t\in[0,T],\quad\text{where}\\
a_n&\coloneqq \sum_{k=1}^N l_k \cdot\1_{\{h_k/l_k<h_n/l_n\}} 
+ \sum_{k=1}^{n-1} l_k \cdot\1_{\{h_k/l_k=h_n/l_n\}},
\qquad n\in\br{N},
\end{split}\end{equation}
is piecewise linear and convex with $F_{l,h}(0)=0$. Differently put, $F_{l,h}$ is linear on each interval $(a_n,a_n+l_n)$ with length $l_n$ and height $h_n$. Moreover, any piecewise linear convex function started at zero whose faces have lengths~$l$ and heights~$h$ must equal $F_{l,h}$.\\
{\nf{(b)}} Suppose $N<\infty$. Given two sequences of heights $h=(h_n)_{n=1}^N$ and $h'=(h'_n)_{n=1}^N$, denote the corresponding functions in~\eqref{eq:CM-formula} by $F_{l,h}$ and $F_{l,h'}$ with sequences $(a_n)_{n=1}^N$ and $(a_n')_{n=1}^N$ of the left endpoints of the intervals on which these functions are linear, respectively. Define the function 
\[
G_{l,h,h'}(t)\coloneqq
    \sum_{n=1}^N h_n\min\{(t-a'_n)^+/l_n,1\},
    \qquad t\in[0,T].
\]
Then, we have 
\begin{equation}
\label{eq:sandwich}
\max\{\|F_{l,h}-F_{l,h'}\|_\infty,
    \|F_{l,h'}-G_{l,h,h'}\|_\infty\}
\leq\max\bigg\{\sum_{n=1}^N (h_n-h_n')^+,\sum_{n=1}^N (h'_n-h_n)^+\bigg\},
\end{equation}
where $\|f\|_\infty\coloneqq\sup_{t\in[0,T]}|f(t)|$ denotes the supremum norm.
\end{lem}

The piecewise linear function $G_{l,h,h'}$ need not be convex. However, it can be easily compared (in all cases, including $N=\infty$) with $F_{l,h'}$, because the intervals of linearity for 
$F_{l,h'}$ and 
$G_{l,h,h'}$
coincide.  
%in the case $N=\infty$. 
The function $G_{l,h,h'}$
will play a key bridging role in the proof of Proposition~\ref{prop:conv_pw_linear} below.

\begin{proof}[Proof of Lemma~\ref{lem:seq-to-cm}]
(a) The lengths of the subintervals $(a_n,a_n+l_n)$, $n\le N$, of $[0,T]$ sum up to $\sum_{n=1}^N l_n=T$. By comparing the respective slopes in the definition of $a_n$, it follows that these intervals are pairwise disjoint. Moreover, $F_{l,h}$ is convex on $[0,T]$ and linear on every $(a_n,a_n+l_n)$. Indeed, since a function is convex if and only if it has a non-decreasing right-derivative a.e., $F_{l,h}$ is convex. Any other piecewise linear convex function with the same faces must have the same derivative as $F_{l,h}$. Furthermore, if such a function also starts at $0$, it must equal $F_{l,h}$. 

(b) A termwise comparison shows that 
\[
-\sum_{n=1}^N (h_n-h'_n)^+
\le F_{l,h'}-G_{l,h,h'}\le
\sum_{n=1}^N (h'_n-h_n)^+,
\]
pointwise. Thus, it remains to show the inequality for $\|F_{l,h}-F_{l,h'}\|_\infty$, which requires two steps.

\textbf{Step 1.} First assume there exists $m\in\br{N}$ such that $h'_m\ne h_m$ and $h_n=h'_n$ for $n\in\br{N}\setminus\{m\}$. By symmetry we may assume $h'_m>h_m$. For all $n\in\br{N}$, define the slopes $s_n\coloneqq h_n/l_n$ and $s'_n\coloneqq h'_n/l_n$. 
Thus $s_m'>s_m$ and, if $n\ne m$, we have $s_n=s_n'$.
Since
\[
a_n=\sum_{k=1}^N l_k\cdot\1_{\{s_k<s_n\}} 
+ \sum_{k=1}^{n-1} l_k\cdot\1_{\{s_k=s_n\}},
\qquad%\text{and}\quad
a'_n=\sum_{k=1}^N l_k\cdot\1_{\{s'_k<s'_n\}} 
+ \sum_{k=1}^{n-1} l_k \cdot\1_{\{s'_k=s'_n\}},
\]
the right-derivatives $f_{l,h}$ and $f_{l,h'}$ of $F_{l,h}$ and $F_{l,h'}$, respectively, are piecewise constant non-decreasing functions satisfying  $f_{l,h}\le f_{l,h'}$ on $[0,T]$. Since $F_{l,h}(0)=0=F_{l,h'}(0)$, we deduce that $F_{l,h'}-F_{l,h}\ge 0$.

By construction, $F_{l,h'}\ge G_{l,h,h'}$ pointwise (in fact, termwise) and  $\|F_{l,h'}-G_{l,h,h'}\|_\infty=h'_m-h_m$. Put $b_n\coloneqq a_n+l_n$ and $b'_n\coloneqq a'_n+l_n$ for $n\in\br{N}$ and note that, since $s_m\le s'_m$, we have $a_m\le a'_m$ and 
\begin{align*}
F_{l,h}(t)&=G_{l,h,h'}(t)=F_{l,h'}(t),
&&\text{for $t\in[0,a_m]$,}\\ 
F_{l,h}(t)&=G_{l,h,h'}(t)\le F_{l,h'}(t)\le G_{l,h,h'}(t)+(h'_m-h_m),
&&\text{for $t\in[b'_m,T]$.}
\end{align*}
Thus, to establish~\eqref{eq:sandwich} in this case, it suffices to prove that $F_{l,h'}(t)-F_{l,h}(t)\le h'_m-h_m$ on $t\in[a_m,b'_m]$. %We will show each inequality separately for the intervals $[a_m,b_m]$ and $[b_m,b'_m]$.

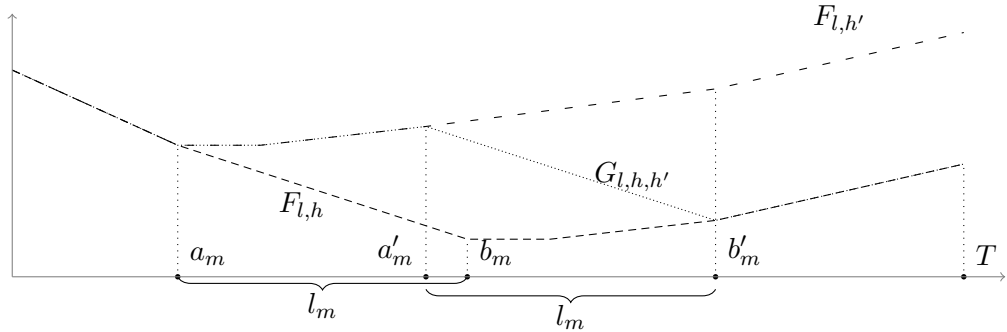
\begin{figure}[ht]
\begin{center}
	\begin{tikzpicture}
		% f_1
		\draw[loosely dashed] (0,5*.25) -- ++(2*1.1,-4*.25) -- ++(1*1.1,0*.25) 
		-- ++(2*1.1,1*.25) -- ++(3.5*1.1,2*.25) --++ (3*1.1,3*.25)% --++ (1*1.1,0)
		node[below] {};
		% g
		\draw[densely dotted] (0,5*.25) -- ++(2*1.1,-4*.25) -- ++(1*1.1,0*.25) 
		-- ++(2*1.1,1*.25) -- ++(3.5*1.1,-5*.25) --++ (3*1.1,3*.25)% --++ (1*1.1,0)
		node[below] {};
		% f_2
		\draw[densely dashed] (0,5*.25) -- ++(2*1.1,-4*.25) -- ++(3.5*1.1,-5*.25) 
		-- ++(1*1.1,0*.25) -- ++(2*1.1,1*.25) --++ (3*1.1,3*.25)% --++ (1*1.1,0)
		node[above] {};
		
		% labels C's and g
		\node[circle, scale=0.5, label=above :{$F_{l,h'}$}] at (10*1.1,6*.25) {}{};
		\node[circle, scale=0.5, label=below :{$G_{l,h,h'}$}] at (7.5*1.1,1*.25) {}{};
		\node[circle, scale=0.5, label=below :{$F_{l,h}$}] at (3.5*1.1,-.6*.25) {}{};
		
		% lambda
		\draw[dotted] (11.5*1.1,-6*.25) -- (11.5*1.1,-0*.25)
		node[above] {};
		\node[circle, fill=black, scale=0.2, label=above right:{$T$}] 
		at (11.5*1.1,-6*.25) {}{};
		
		% a_{m'}
		\draw[dotted] (2*1.1,-6*.25) -- (2*1.1,1*.25)
		node[above] {};
		\node[circle, fill=black, scale=0.2, label=above right:{$a_m$}] 
		at (2*1.1,-6*.25) {}{};
		
		% b_m
		\draw[dotted] (5.5*1.1,-6*.25) -- (5.5*1.1,-4*.25)
		node[above] {};
		\node[circle, fill=black, scale=0.2, label=above right:{$b_m$}] 
		at (5.5*1.1,-6*.25) {}{};
		
		% b'_{m}
		\draw[dotted] (8.5*1.1,-6*.25) -- (8.5*1.1,4*.25)
		node[above] {};
		\node[circle, fill=black, scale=0.2, label=above right:{$b'_m$}] 
		at (8.5*1.1,-6*.25) {}{};
		
		% a'_m
		\draw[dotted] (5*1.1,-6*.25) -- (5*1.1,2*.25)
		node[above] {};
		\node[circle, fill=black, scale=0.2, label=above left:{$a'_m$}] 
		at (5*1.1,-6*.25) {}{};
		
		\draw [decorate,decoration={brace,amplitude=5pt,mirror,raise=1ex}]
		(5*1.1,-6*.25) -- (8.5*1.1,-6*.25) node[midway,yshift=-1.4em]{$l_m$};
		
		\draw [decorate,decoration={brace,amplitude=5pt,mirror,raise=0ex}]
		(2*1.1,-6*.25) -- (5.5*1.1,-6*.25) node[midway,yshift=-1em]{$l_m$};
		
		% Draw the axis
		\draw [thin, draw=gray, ->] (0,-6*.25) -- (12*1.1,-6*.25);
		\draw [thin, draw=gray, ->] (0,-6*.25) -- (0,8*.25);
		
	\end{tikzpicture}
	\caption{\small Comparison between $F_{l,h}$, $F_{l,h'}$ and $G_{l,h,h'}$.
	}\label{fig:one-swap}
\end{center}
\end{figure}

By construction of $a'_m$, the right-derivative $f_{l,h'}$ is smaller or equal to $s'_m=h'_m/l_m$ on $(a_m,b'_m)$. Since $F_{l,h'}(a_m)=F_{l,h}(a_m)$, for $t\in[a_m,b_m]$ we have 
\[
F_{l,h'}(t)-F_{l,h}(t)
=\int_{a_m}^t f_{l,h'}(u)du - s_m (t-a_m)
\le (s'_m - s_m)(t-a_m)
\le (s'_m - s_m)l_m 
= h'_m-h_m.
\]

For $t\in[b_m,b'_m]$ we have $t-l_m\in[a_m,a'_m]$ and thus $F_{l,h}(t)-h_m=G(t-l_m)=F_{l,h'}(t-l_m)$. Hence 
\[
F_{l,h'}(t)-F_{l,h}(t)
=F_{l,h'}(t) - F_{l,h'}(t-l_m) - h_m
=\int_{t-l_m}^t f_{l,h'}(u) du - h_m
\le\int_{t-l_m}^t s'_m du - h_m
= h'_m - h_m.
\]
Thus, $F_{l,h'}-F_{l,h}\le h'_m - h_m$ on $[a_m,b'_m]$, proving~\eqref{eq:sandwich} in this case. 

\textbf{Step 2.} Consider the general case. For $k\in\{0,\ldots,N\}$, let  $h^{(k)}=(h^{(k)}_n)_{n\in\br{N}}$ be given by $h^{(k)}_n\coloneqq h_n\cdot\1_{\{n>k\}}+h'_n\cdot\1_{\{n\le k\}}$ for $n\in\br{N}$. Note that $h'=h^{(0)}$ and $h=h^{(N)}$. Since the sequences $h^{(k)}$ and $h^{(k-1)}$ only differ in the coordinate $h^{(k)}_k\neq h^{(k-1)}_k$, the identity $F_{l,h'}-F_{l,h} =\sum_{k=1}^N (F_{l,h^{(k-1)}}-F_{l,h^{(k)}})$ and \textbf{Step 1} imply~\eqref{eq:sandwich}, completing the proof.
\end{proof}

\begin{lem}
\label{lem:conv_pw_linear_finite}
Let $(N_k)_{k\in\N}$ be a sequence in $\N$  with a limit $N_k\to N_\infty\in\N$. For each $j\in\ov\N$, let $(l_{j,n})_{n\in\br{N_j}}$ be positive numbers satisfying $\sum_{n=1}^{N_j}l_{j,n}=T$, $(h_{j,n})_{n\in\br{N_j}}$ real numbers and $C_j$ the piecewise linear convex function defined in~\eqref{eq:CM-formula} %Lemma~\ref{lem:seq-to-cm}(a) 
with lengths $(l_{j,n})_{j\in\br{N_j}}$ and heights $(h_{j,n})_{j\in\br{N_j}}$. Suppose $l_{k,n}\to l_{\infty,n}$ and $h_{k,n}\to h_{\infty,n}$ as $k\to\infty$ for all $n\in\br{N_\infty}$. Then $\|C_\infty-C_k\|_\infty\to 0$ as $k\to\infty$.
\end{lem}

\begin{proof}
The convergence $N_k\to N_\infty$ as $k\to\infty$ implies $N_k=N_\infty=N$ for all sufficiently large $k$. Thus, we assume without loss of generality that $N_j=N$ for all $j\in\ov\N$. Define $s_{j,n}\coloneqq h_{j,n}/l_{j,n}$ for $j\in\ov\N$ and $n\in\br{N}$ and note that $s_{k,n}\to s_{\infty,n}$ as $k\to\infty$ for all $n\in\br{N}$. Thus, for all sufficiently large $k$, if the inequality $s_{\infty,n}<s_{\infty,m}$ holds, then $s_{k,n}<s_{k,m}$. Thus, we assume this property holds for all $k\in\N$. Moreover, we assume without loss of generality, by relabeling if necessary, that $s_{\infty,1}\le \cdots\le s_{\infty,N}$. 

We will next introduce a sequence of convex functions $F_k$ satisfying the limits $\|C_\infty-F_k\|_\infty\to0$ and $\|F_k-C_k\|_\infty\to 0$ as $k\to\infty$. These convex functions will replace each ``block'' of faces of $C_k$ with a given common \textit{limiting} slope, with a single face with the mean slope.

Let $M\le N$ be the number of distinct slopes in  $\{s_{\infty,n}:n\in\br{N}\}$ and note that $s_{\infty,i_1}<\cdots<s_{\infty,i_M}$, where we set $i_1\coloneqq 1$ and  $i_{n+1}\coloneqq\min\{m\in\{i_n+1,\ldots,N\}:s_{\infty,m}>s_{\infty,i_n}\}$ for $n\in\br{M-1}$. Note that $s_{k,m}\to s_{\infty,i_n}$ as $k\to\infty$ for all $m\in\{i_n,\ldots,i_{n+1}-1\}$. Let $L_{j,n}\coloneqq\sum_{m=i_n}^{i_{n+1}-1}l_{j,m}$ and $H_{j,n}\coloneqq\sum_{m=i_n}^{i_{n+1}-1}h_{j,m}$ for $n\in\br{M}$ and $j\in\ov\N$, where $i_{M+1}\coloneqq N+1$. Furthermore, for $j\in\ov\N$, let $(a_{j,n})_{n\in\br{N}}$ be the left endpoints of the intervals in~\eqref{eq:CM-formula} on which $C_j$ is linear. Note that $C_\infty$ admits the representation $C_\infty(t)=\sum_{n=1}^M H_{\infty,n}\min\{(t-a_{\infty,i_n})^+/L_{\infty,n},1\}$ for $t\in[0,T]$ and define the convex functions $F_k(t)\coloneqq\sum_{n=1}^M H_{k,n}\min\{(t-a_{k,i_n})^+/L_{k,n},1\}$ for $k\in\N$. 
The limits $l_{k,n}\to l_{\infty,n}$ and $h_{k,n}\to h_{\infty,n}$ imply  $a_{k,i_n}\to a_{\infty,i_n}$, $L_{k,n}\to L_{\infty,n}$ and $H_{k,n}\to H_{\infty,n}$ as $k\to\infty$ for $n\in\br{M}$. Thus, we have the pointwise (in fact, termwise) convergence $F_k\to C_\infty$. Since the functions are convex, the pointwise convergence implies $\|C_\infty-F_k\|_\infty\to 0$ as $k\to\infty$. 

To prove that $\|F_k-C_k\|_\infty\to 0$ as $k\to\infty$, note that for $a,c\in\R$ and $b,d>0$ satisfying $a/b\le c/d$, we have $a/b\le (a+c)/(b+d)\le c/d$. Thus, $H_{k,n}/L_{k,n}$ lies between the smallest and largest values of $S_{k,n}\coloneqq\{h_{k,i_n}/l_{k,i_n},\ldots,h_{k,i_{n+1}-1}/l_{k,i_{n+1}-1}\}$. Since all the slopes in $S_{k,n}$ converge to $s_{\infty,i_n}$, by the triangle inequality, we have $\max_{s\in S_{k,n}}|H_{k,n}/L_{k,n}-s| \le\max_{s,s'\in S_{k,n}}|s'-s|\le  b_k\coloneqq 2\max_{m\in\br{N}}|s_{k,m}-s_{\infty,m}|\to0$ as $k\to\infty$. Hence, the right-derivative of $F_k$ is at most $b_k$ away from the right-derivative of $C_k$, implying  $\|F_k-C_k\|_\infty\le b_kT\to 0$ as $k\to\infty$, completing the proof.
\end{proof}

\begin{prop}
\label{prop:conv_pw_linear}
Let $N_k$ and $N_\infty$ be $\ov\N$-valued random variables with $N_k\to N_\infty$ a.s. as $k\to\infty$. Let $(l_{j,n})_{n=1}^{N_j}$, $j\in\ov\N$, be random sequences of positive numbers satisfying $\sum_{n=1}^{N_j} l_{j,n}=T$ and $(h_{j,n})_{n=1}^{N_j}$, $j\in\ov\N$, sequences of random variables with $\sum_{n=1}^{N_j}|h_{j,n}|<\infty$ a.s. Let $C_j$ be the piecewise linear convex function in~\eqref{eq:CM-formula}
%in Lemma~\ref{lem:seq-to-cm}(a) 
with sequences of lengths $(l_{j,n})_{n=1}^{N_j}$ and heights $(h_{j,n})_{n=1}^{N_j}$ for $j\in\ov\N$. Suppose $l_{k,n}\to l_{\infty,n}$ a.s. and $h_{k,n}\to h_{\infty,n}$ a.s. as $k\to\infty$ for all $n< N_\infty+1$ and
\begin{equation}
\label{eq:tightness}
\lim_{M\to\infty}\limsup_{k\to\infty}
    \e\min\bigg\{1,\sum_{n=M}^{N_k} |h_{k,n}|\bigg\}=0.
\end{equation}
Then $\|C_\infty-C_k\|_\infty\cip 0$ as $k\to\infty$.
\end{prop}

\begin{proof}
On the event $\{N_\infty<\infty\}$, by Lemma~\ref{lem:conv_pw_linear_finite} we have
$\|C_\infty-C_k\|_\infty\to 0$ a.s. as $k\to\infty$
(and~\eqref{eq:tightness} holds by our summing convention).
%By partitioning the probability space, we may assume that %$N_\infty\in\ov\N$ is deterministic. Moreover, the case %$N_\infty<\infty$ follows from .
Assume we are on the event $\{N_\infty=\infty\}$. For each $M\in\N$ and $j\in\ov\N$, let $C_{j,M}$ be the piecewise linear convex function in~\eqref{eq:CM-formula}
%in Lemma~\ref{lem:seq-to-cm}(a) 
with lengths $(l_{j,n})_{n=1}^{N_j}$ and heights $(h_{j,n}\1_{\{n<M\}})_{n=1}^{N_j}$. 
Recall $a\wedge b=\min\{a,b\}$ for any $a,b\in \R$.
For each $j\in\ov\N$, define
\[
a_{j,n}\coloneqq \sum_{m=1}^{N_j} l_{j,m} \cdot\1_{\{h_{j,m}/l_{j,m}<h_{j,n}/l_{j,n}\}} 
+ \sum_{m=1}^{n-1} l_{j,m} \cdot\1_{\{h_{j,m}/l_{j,m}=h_{j,n}/l_{j,n}\}},
\qquad n\in\br{N_j},%\enskip j\in\ov\N.
\]
%Define for each $j\in\ov\N$ 
end the function $G_{j,M}(t)\coloneqq\sum_{m=1}^{N_j\wedge (M-1)}h_{j,n}\min\{(t-a_{j,n})^+/l_{j,n},1\}$, $t\in[0,T]$. Note that $C_j$ and $G_{j,M}$ are linear on every interval $(a_{j,n},a_{j,n}+l_{j,n})$, $n\in\br{N_j}$, but $C_{j,M}$ may have different intervals of linearity. Since $1\wedge(x+y)\le 1\wedge x+ 1\wedge y$ for all $x,y\ge0$, the triangle inequality implies 
\begin{equation}
\label{eq:5-triangle}
1\wedge\|C_{\infty}-C_{k}\|_\infty
\le A_{\textbf{(I)}}  + A_{\textbf{(II)}} + A_{\textbf{(III)}} + A_{\textbf{(IV)}} + A_{\textbf{(V)}},
\end{equation}
where $A_{\textbf{(I)}}\coloneqq 1\wedge\|C_{\infty}-G_{\infty,M}\|_\infty$, $A_{\textbf{(II)}}\coloneqq1\wedge\|G_{\infty,M}-C_{\infty,M}\|_\infty$, $A_{\textbf{(III)}}\coloneqq1\wedge\|C_{\infty,M}-C_{k,M}\|_\infty$, $A_{\textbf{(IV)}}\coloneqq1\wedge\|C_{k,M}-G_{k,M}\|_\infty$ and $A_{\textbf{(V)}}\coloneqq1\wedge\|G_{k,M}-C_{k}\|_\infty$.
As $\zeta_n\cip 0$ as $n\to\infty$ if and only if $\e[1\wedge|\zeta_n|]\to 0$, it suffices to prove that the expectation of each of the terms in~\eqref{eq:5-triangle} converges to $0$ as we take $\limsup_{k\to\infty}$ and then $M\to\infty$.

\textbf{(I)}\&\textbf{(V)}. 
By construction of $C_j$ and $G_{j,M}$ we have $\|C_j-G_{j,M}\|_\infty \le\sum_{n=M}^{N_j}|h_{j,n}|$ for all $j\in\ov\N$. Thus, $\|C_\infty-G_{\infty,M}\|_\infty\to 0$ a.s. and hence $\e A_{\textbf{(I)}}=\e[1\wedge\|C_\infty-G_{\infty,M}\|_\infty]\to 0$ as $M\to\infty$. Moreover, by assumption in~\eqref{eq:tightness}, we have
\[
\limsup_{k\to\infty}\e A_{\textbf{(V)}}
=\limsup_{k\to\infty}\e[1\wedge\|C_k-G_{k,M}\|_\infty]
\le \limsup_{k\to\infty}
    \e\min\bigg\{1,\sum_{n=M}^{N_k}|h_{k,n}|\bigg\}
\xrightarrow[M\to\infty]{}0.
\]

\textbf{(III)}. 
For all $j\in\ov\N$, 
the faces of $C_{j,M}$
corresponding to $n\in\br{N_j}\setminus\br{M-1}$ are horizontal.
By convexity, we may assume they lie next to each other in the graph
of $C_{j,M}$. Merging all the lengths $l_{j,n}$, $n\in\br{N_j}\setminus\br{M-1}$, yields a representation of $C_{j,M}$ with $N_j\wedge M$ faces.
Fix $M\in\N$. Lemma~\ref{lem:conv_pw_linear_finite} yields $\|C_{\infty,M}-C_{k,M}\|_\infty\to 0$ a.s. and thus $\e A_{\textbf{(III)}}=\e[1\wedge\|C_{\infty,M}-C_{k,M}\|_\infty]\to 0$ as $k\to\infty$.

\textbf{(II)}\&\textbf{(IV)}. 
The idea is to apply~\eqref{eq:sandwich} in Lemma~\ref{lem:seq-to-cm}(b) to bound $\|C_{j,M}-G_{j,M}\|_\infty$, with $F_{l,h}$, $G_{l,h,h'}$ and $F_{l,h'}$ in Lemma~\ref{lem:seq-to-cm}(b) given by $C_{j,M}$, $G_{j,M}$ and $F_{j,M}$, respectively. The piecewise linear convex function $F_{j,M}$, which shares the intervals of linearity with those of $G_{j,M}$, is yet to be defined. 

Note that $G_{j,M}$ possesses a piecewise linear representation with at most $2M$ faces. Indeed, $G_{j,M}$ is linear on  $(a_{j,n},a_{j,n}+l_n)$, $n\in\br{N_j\wedge (M-1)}$, and the complement  $(0,T)\setminus\bigcup_{n=1}^{N_j\wedge (M-1)}[a_{j,n},a_{j,n}+l_{j,n}]$ is a disjoint union of $M_j\le M+1$ open intervals, say $(a'_{j,n},a'_{j,n}+l'_{j,n})$, $n\in\br{M_j}$. For each $n\in\br{M_j}$, define the height $h'_{j,n}\coloneqq\sum_{m\in S_{j,n}}h_{j,m}$, where $S_{j,n}\coloneqq\{m\in\br{N_j}\setminus\br{M-1}: 
a_{j,m}\in (a'_{j,n},a'_{j,n}+l'_{j,n})\}$. Put differently, the height $h'_{j,n}$ equals the sum of all the heights of the faces of $C_j$ that lie above the interval $[a'_{j,n},a'_{j,n}+l'_{j,n}]$. For any $j\in\ov\N$, define the function 
\begin{equation}
\label{eq:F_jM}
F_{j,M}(t)
\coloneqq
\sum_{n=1}^{N_j\wedge(M-1)}
    h_{j,n}\min\{(t-a_{j,n})^+/l_{j,n},1\}
+\sum_{n=1}^{M_j}
    h'_{j,n}\min\{(t-a'_{j,n})^+/l'_{j,n},1\},
\quad t\in[0,T].
\end{equation}

We will show that $F_{j,M}$ is convex. It suffices to prove that the consecutive slopes of $F_{j,M}$ on adjacent intervals of linearity increase. If the consecutive intervals are $(a_{j,m},a_{j,m}+l_{j,m})$ and $(a_{j,n},a_{j,n}+l_{j,n})$ (i.e. they come from the first sum in~\eqref{eq:F_jM}), then by construction the intervals must be adjacent with the same slopes in the convex function $C_j$, implying the corresponding slopes satisfy the correct ordering. Assume the consecutive intervals are $(a_{j,m},a_{j,m}+l_{j,m})$ and $(a'_{j,n},a'_{j,n}+l'_{j,n})$ (i.e. the first interval comes from first sum and the second interval comes from the second sum in~\eqref{eq:F_jM}). Suppose $a_{j,m}=a'_{j,n}+l'_{j,n}$ and note that, for $a,c\in\R$ and $b,d>0$ with $a/b\le c/d$ we have $a/b\le(a+c)/(b+d) \le c/d$. Thus, by definition of $h'_{j,n}$, we have 
$%\inf_{i\in S_{j,n}}h_{j,i}/l_{j,i}
%\le 
h'_{j,n}/l'_{j,n}
\le\sup_{i\in S_{j,n}}h_{j,i}/l_{j,i}\le h_{j,m}/l_{j,m}$, where the last inequality holds because 
$a_{j,m}=a'_{j,n}+l'_{j,n}$ and
$C_j$ is convex. The case $a_{j,n}'=a_{j,m}+l_{j,n}$ is analogous since the slope $h'_{j,n}/l'_{j,n}$ is a mean of slopes at least as large as $h_{j,m}/l_{j,m}$, implying
the convexity of $F_{j,M}$.

Define $l=(l_{j,1},\ldots,l_{j,N_j\wedge(M-1)}, l'_{j,1},\ldots,l'_{j,M_j})$,  $h=(h_{j,1},\ldots,h_{j,N_j\wedge(M-1)}, h'_{j,1},\ldots,h'_{j,M_j})$ and $h'=(h_{j,1},\ldots,h_{j,N_j\wedge(M-1)}, 0,\ldots,0)$.
Note that the corresponding functions $F_{l,h}$, $F_{l,h'}$ and $G_{l,h,h'}$ in~\eqref{eq:CM-formula} %Lemma~\ref{lem:seq-to-cm}(b) 
equal $F_{j,M}$, $C_{j,M}$ and $G_{j,M}$, so~\eqref{eq:sandwich} implies the inequality 
\[
\|G_{j,M}-C_{j,M}\|_\infty
\le \|G_{j,M}-F_{j,M}\|_\infty + \|F_{j,M}-C_{j,M}\|_\infty
\le 2\sum_{m=1}^{M_j}|h'_{j,m}|
\le 2\sum_{n=M}^{N_j}|h_{j,n}|.
\]
Thus, $\|G_{\infty,M}-C_{\infty,M}\|_\infty\to 0$ a.s. and hence $\e A_{\textbf{(II)}}=\e[1\wedge\|C_\infty-G_{\infty,M}\|_\infty]\to 0$ as $M\to\infty$. Moreover, by assumption in~\eqref{eq:tightness}, we have
\[
\limsup_{k\to\infty}\e A_{\textbf{(IV)}}
=\limsup_{k\to\infty}\e[1\wedge\|G_{k,M}-C_{k,M}\|_\infty]
\le 2\limsup_{k\to\infty}
    \e\min\bigg\{1,\sum_{n=M}^{N_k}|h_{k,n}|\bigg\}
\xrightarrow[M\to\infty]{}0.\qedhere
\]
\end{proof}

\subsection{The convex minorant of random walks}
\label{subsec:CM-RW}

Let a function $f:[0,T]\to\R$ satisfy $f(0)=0$. Given parameters $0\le g\le u\le d\le T$, the \emph{3214 transformation}, introduced in~\cite{MR2825583}, is defined by 
\[
\Theta_{g,u,d} f(t)
=\begin{cases}
f(u+t) - f(u), 
    & 0\le t\le d-u,\\
f(d) - f(u) + f(g+t-(d-u)) - f(g), 
    & d-u< t\le d-g,\\
f(d) - f(t-(d-g)), 
    & d-g< t\le d,\\
f(t), 
    & d<t.
\end{cases}
\]

The  3214 transformation reorders the segments of the graph of $f$ as follows: the segments
(I) $[0,g]$, (II) $[g,u]$, (III) $[u,d]$ and (IV) $[d,T]$ are moved to (III) $[0,d-u]$, (II) $[d-u,d-g]$, (I) $[d-g,d]$ and (IV) $[d,T]$, respectively (see also Figure~\ref{fig:3214-CM} below).
This transformation possesses the following remarkable property when applied to continuous piecewise linear functions with a given set of increments.

\begin{prop}[{\cite[Thm~1]{MR2825583}}]
\label{prop:3214-RW} 
Fix $n\in\N$ and let $x_1,\ldots,x_n$ be real numbers, such that no two subsets have the same mean. Let $\fl{y}:=\max\{m\in\Z : m\le y\}$, $y\in\R$, and  $\pi:\br{n}\to\br{n}$ be a uniform random permutation. Define the piecewise linear random function $R=(R(t))_{t\in[0,T]}$ by $R(T):=\sum_{k=1}^n x_k$ and 
\begin{equation}
\label{eq:poly-RW}
R(t)\coloneqq \left(nt/T-\fl{nt/T}\right)
    x_{\pi(\fl{nt/T}+1)} +\sum_{k =1}^{\fl{nt/T}}x_{\pi(k)}, \qquad t\in[0,T).
\end{equation}
Let  $\CM{T}{R}$ denote the convex minorant of $R$
and let $W\sim\U(0,T)$ be independent of $R$.
Let
$0=V_0<\cdots<V_N=T$
be the sequence of contact points between the piecewise linear functions $R$ and $C_T^R$ and $j\in\br{N}$ the unique index such that $W\in (V_{j-1},V_j]$. Define $U\coloneqq\cl{W n/T}T/n$, $G\coloneqq V_{j-1}$ and $D\coloneqq V_j$. Then the 3214 transform with parameters $(G,U,D)$ satisfies the identity in law
\[
(U,R)\eqd (D-G,\Theta_{G,U,D} R).
\]
\end{prop}

%\begin{rem}
%\label{rem:3214-RW}
%Note that if $\vartheta_1,\ldots,\vartheta_n$ are iid continuous random variables, then no two subsets of the set $S=\{\vartheta_1,\ldots,\vartheta_n\}$ have the same mean and conditional on $S$, the values of the sequence appear as given by a uniform permutation $\pi\indep S$. In short, Theorem~\ref{prop:3214-RW} applies to the case when the random variables used to construct the random walk $R$ are iid and continuous.
%\end{rem}

We recall below a proof of Proposition~\ref{prop:3214-RW} based on a simple argument from~\cite{MR2831081}.

\begin{thm}
\label{thm:SB-representation-RW}
Let $x_1,\ldots,x_n$ be arbitrary real numbers and $\pi:\br{n}\to\br{n}$ a uniform random permutation. Define $R$ by~\eqref{eq:poly-RW} and let $(V_k)_{k\in\N}$ be an iid sequence of $\U(0,1)$ random variables independent of $\pi$. Define recursively $L_{n,0}\coloneqq T$, $L_{n,k}\coloneqq\lfloor L_{n,k-1}V_kn/T\rfloor T/n$, $\ell_{n,k}\coloneqq L_{n,k-1}-L_{n,k}$ for $k\in\N$ and let $N\le n$ be the largest integer for which $\ell_{n,N}>0$. Then the convex minorant $\CM{T}{R}$ has the same law as the piecewise linear convex function defined in~\eqref{eq:CM-formula} with sequences of lengths $(\ell_{n,k})_{k=1}^N$ and heights $(R(L_{n,k-1})-R(L_{n,k}))_{k=1}^N$. 
\end{thm}

We stress that in Theorem~\ref{thm:SB-representation-RW}, the reals $x_1,\ldots,x_n$ may have multiple subsets with the same mean. Our proof approximates a general sequence by one satisfying the ``no ties'' assumption of Proposition~\ref{prop:3214-RW} and
applies a convergence result for piecewise linear
convex functions from Section~\ref{subsec:CM-PL}.
The proof of Theorem~\ref{thm:SB-representation-RW} in~\cite{MR2825583} sub-samples the ties, resulting in a more involved statement of the theorem.

%Moreover, in Theorem~\ref{thm:SB-representation-RW}, we may instead provide directly the random variables $x_{\pi(1)},\ldots,x_{\pi(n)}$ with an arbitrary exchangeable distribution.

\begin{proof}[Proof of Theorem~\ref{thm:SB-representation-RW}]
First assume that no two subsets of the numbers $x_1,\ldots,x_n$ have the same mean. Let $\pi$ and $(G,U,D)$ be as in Proposition~\ref{prop:3214-RW}. By  Proposition~\ref{prop:3214-RW}, the face decomposition of $\CM{T}{R}$ contains the face with length-height pair $(D-G,\CM{T}{R}(D)-\CM{T}{R}(G))$, which has the same law as $(U,R(U))$, and the faces of a copy of $\CM{T-(D-G)}{R}$ independent of the first face. Indeed, this copy is in fact the convex minorant of $\Theta_{G,U,D}R$ on $[T-(D-G),T]$ and we may apply the same procedure to this copy. Iterating this procedure, we obtain a (finite) sequence of lengths of the faces of $\CM{T}{R}$, which has the same law as the sequence $(\ell_{k,n})_{k=1}^N$, and the corresponding heights, which have the same law as $(R(L_{k-1,n})-R(L_{k,n}))_{k=1}^N$, completing the proof in this case. 
%In fact, these faces are ``discovered'' by an independent %sequence of uniform random variables.

To prove the general case, we recall that $\|\CM{T}{x}-\CM{T}{y}\|_\infty\le\|x-y\|_\infty$ for any bounded functions $x,y:[0,T]\to\R$. Indeed, this follows from the fact that $\CM{T}{x}-\|x-y\|_\infty$ is convex and $\CM{T}{x}-\|x-y\|_\infty\le x-\|x-y\|_\infty\le y$ pointwise. For any $\ve>0$ consider real numbers $x_{1,\ve},\ldots,x_{n,\ve}$ such that no two subsets have the same mean and $\sum_{k=1}^n|x_k-x_{k,\ve}|\le\ve$. Let $R_{\ve}$ be the corresponding random function in~\eqref{eq:poly-RW} (with the same permutation $\pi$). Note that $\|\CM{T}{R}-\CM{T}{R_{\ve}}\|_\infty\le\|R-R_{\ve}\|_\infty\le\ve\to 0$ as $\ve\to 0$. Moreover, by the argument in the previous paragraph, $\CM{T}{R_\ve}$ has the same law as the piecewise linear convex function $C_{\ve}$ given by~\eqref{eq:CM-formula} with lengths $(\ell_{n,k})_{k=1}^N$ and heights $(R_{\ve}(L_{n,k-1})-R_\ve(L_{n,k}))_{k=1}^N$. Let $C$ be the piecewise linear convex function given by~\eqref{eq:CM-formula} with lengths $(\ell_{n,k})_{k=1}^N$ and heights $(R(L_{n,k-1})-R(L_{n,k}))_{k=1}^N$.  Lemma~\ref{lem:conv_pw_linear_finite} yields $\|C-C_{\ve}\|_\infty\to0$ a.s. as $\ve\to0$, implying $C\eqd\CM{T}{R}$ and completing the proof.
\end{proof}

The proof of Proposition~\ref{prop:3214-RW} requires the following  lemma.

\begin{lem}
\label{lem:RW-above-line}
Let $x_1,\ldots,x_n$ be real numbers such that no two subsets have the same mean. Then there is a unique $k^*\in\br{n}$ such that $\sum_{i=1}^k x_{(k^*+i)\Mod{n}}\ge \frac{k}{n}\sum_{i=1}^n x_i$ for all $k\in\br{n}$, i.e. the walk with increments $x_{(k^*+1)\Mod{n}},\ldots,x_{(k^*+n)\Mod{n}}$ is above the line connecting zero with the endpoint $\sum_{i=1}^n x_i$. 
\end{lem}

\begin{proof}
Define $s\coloneqq\sum_{i=1}^n x_i/n$. If the walk $k\mapsto\sum_{i=1}^k (x_i-s)$, $k\in\br{n}$, attained its minimum at two times $k_1<k_2$, then $\sum_{i=k_1+1}^{k_2}x_i/(k_2-k_1)= s$, contradicting the assumption. It is easily seen that the $k^*$ in the statement of the lemma is the time at which this walk attains its minimum on $\br{n}$.
\end{proof}

\begin{proof}[Proof of Proposition~\ref{prop:3214-RW} {(\cite{MR2831081})}]
Note that, if a random element $\zeta$ is uniformly distributed in some finite set $\mathcal{Z}$ and if the map $\varphi:\mathcal{Z}\to\mathcal{Z}$ is injective (and thus bijective), then $\varphi(\zeta)$ is also uniformly distributed on $\mathcal{Z}$. Thus, since $\pi$ and $U$ are uniform and independent, it is sufficient to show that the transformation $(u,f)\mapsto (d-g,\Theta_{g,u,d}f)$ is injective. 

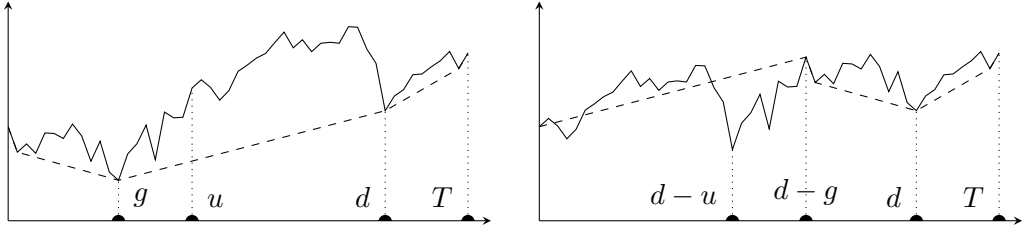
\begin{figure}[ht]
\begin{center}
\begin{tikzpicture}
\begin{axis} 
	[
	ymin=-.75,
	ymax=1,
	xmin=0,
	xmax=1.05,
	width=8cm,
	height=4.5cm,
	axis on top=true,
	axis x line=bottom, 
	axis y line=middle,
	ticks = none,
	legend style={at={(1.03,.865)},anchor=north east}]
	% T
	\node[circle, fill=black, scale=0.5, label=above left:{$T$}] at (100,-.75) {}{};
	\addplot[dotted]
	coordinates {(1,-.75)(1,0.588)};
	% g
	\node[circle, fill=black, scale=0.5, label=above right:{$g$}] at (24,-.75) {}{};
	\addplot[dotted]
	coordinates {(.24,-.75)(.24,-0.427)};
	% u
	\node[circle, fill=black, scale=0.5, label=above right:{$u$}] at (40,-.75) {}{};
	\addplot[dotted]
	coordinates {(.40,-.75)(.40,0.309)};
	% d
	\node[circle, fill=black, scale=0.5, label=above left:{$d$}] at (82,-.75) {}{};
	\addplot[dotted]
	coordinates {(.82,-.75)(.82,.128)};
	
	%\addplot[dotted]
	%coordinates {(0,0)(.406,-0.61)(.8,0.599)(.936,.403)(1.0,0.625)};
	
	% Draw f
	\addplot[
	solid, color=black, thin,
	]
	coordinates {
(0.0,0.0)(0.02,-0.203)(0.04,-0.137)(0.06,-0.212)(0.08,-0.0494)(0.1,-0.0537)(0.12,-0.0938)(0.14,0.0214)(0.16,-0.0741)(0.18,-0.275)(0.2,-0.116)(0.22,-0.362)(0.24,-0.427)(0.26,-0.217)(0.28,-0.136)(0.3,0.0125)(0.32,-0.263)(0.34,0.117)(0.36,0.071)(0.38,0.0756)(0.4,0.309)(0.42,0.375)(0.44,0.314)(0.46,0.212)(0.48,0.288)(0.5,0.442)(0.52,0.492)(0.54,0.548)(0.56,0.585)(0.58,0.672)(0.6,0.755)(0.62,0.63)(0.64,0.694)(0.66,0.599)(0.68,0.666)(0.7,0.674)(0.72,0.666)(0.74,0.798)(0.76,0.794)(0.78,0.62)(0.8,0.506)(0.82,0.128)(0.84,0.245)(0.86,0.296)(0.88,0.415)(0.9,0.421)(0.92,0.477)(0.94,0.527)(0.96,0.6)(0.98,0.464)(1.0,0.588)
	};
	
	% Draw CM(f)
	\addplot[
	dashed, color=black, thin,
	]
	coordinates {
(0.0,0.0)(0.02,-0.203)(0.24,-0.427)(0.82,0.128)(0.98,0.464)(1.0,0.588)
	};
\end{axis}
\end{tikzpicture}\hspace{.5cm}
\begin{tikzpicture}
\begin{axis} 
	[
	ymin=-.75,
	ymax=1,
	xmin=0,
	xmax=1.05,
	width=8cm,
	height=4.5cm,
	axis on top=true,
	axis x line=bottom, 
	axis y line=middle,
	ticks = none,
	legend style={at={(1.03,.865)},anchor=north east}]
	% T
	\node[circle, fill=black, scale=0.5, label=above left:{$T$}] at (100,-.75) {}{};
	\addplot[dotted]
	coordinates {(1,-.75)(1,0.588)};
	% d-u
	\node[circle, fill=black, scale=0.5, label=above left:{$d-u$}] at (42,-.75) {}{};
	\addplot[dotted]
	coordinates {(.42,-.75)(.42,-0.181)};
	% d-g
	\node[circle, fill=black, scale=0.5, label=above :{$d-g$}] at (58,-.75) {}{};
	\addplot[dotted]
	coordinates {(.58,-.75)(.58,0.555)};
	% d
	\node[circle, fill=black, scale=0.5, label=above left:{$d$}] at (82,-.75) {}{};
	\addplot[dotted]
	coordinates {(.82,-.75)(.82,.128)};
	
	%\addplot[dotted]
	%coordinates {(0,0)(.136,-0.196)(.53,1.013)(.936,.403)(1.0,0.625)};
	
	% Draw transform of f
	\addplot[
	solid, color=black, thin,
	]
	coordinates {
(0.0,0.0)(0.02,0.0661)(0.04,0.00577)(0.06,-0.0963)(0.08,-0.0205)(0.1,0.133)(0.12,0.183)(0.14,0.24)(0.16,0.276)(0.18,0.363)(0.2,0.446)(0.22,0.321)(0.24,0.385)(0.26,0.29)(0.28,0.357)(0.3,0.366)(0.32,0.357)(0.34,0.489)(0.36,0.486)(0.38,0.311)(0.4,0.198)(0.42,-0.181)(0.44,0.0293)(0.46,0.111)(0.48,0.259)(0.5,-0.0165)(0.52,0.363)(0.54,0.317)(0.56,0.322)(0.58,0.555)(0.6,0.353)(0.62,0.418)(0.64,0.343)(0.66,0.506)(0.68,0.501)(0.7,0.461)(0.72,0.577)(0.74,0.481)(0.76,0.28)(0.78,0.439)(0.8,0.193)(0.82,0.128)(0.84,0.245)(0.86,0.296)(0.88,0.415)(0.9,0.421)(0.92,0.477)(0.94,0.527)(0.96,0.6)(0.98,0.464)(1.0,0.588)
	};
	
	% Draw transform of CM(f)
	\addplot[
	dashed, color=black, thin,
	]
	coordinates {
(0.0,0.0)(0.4,0.383)(0.56,0.536)(0.58,0.555)(0.6,0.353)(0.82,0.128)(0.98,0.464)(1.0,0.588)
	};
\end{axis}
\end{tikzpicture}
\caption{The pictures show a path of a random walk $R$ (\emph{solid}) and its convex minorant $C_T^{R}$ (\emph{dashed}) on $[0,T]$ on the left and their 3214 transforms on the right. The transform is associated to some $u\in(0,T)$ and the endpoints $\{g,d\}$ of the maximal face of $C_T^R$ containing $u$.
}\label{fig:3214-CM}
\end{center}
\end{figure}

Assume without loss of generality that $T=n$. 
To prove the injectivity, it suffices to describe the inverse transformation. Given $d-g$, and $\ti f \coloneqq\Theta_{g,u,d}f$, note that $d-u$ is the unique time in Lemma~\ref{lem:RW-above-line} for the increments of $\ti f$ over the set $\br{d-g}$, %i.e. it is the unique time such that $x\mapsto \ti f(\min\{x-(d-u),d-g\}) + \ti f(\max\{x-(d-g),0\})$ remains above the line $x\mapsto x\ti f(d-g)/(d-g)$, 
see Figure~\ref{fig:3214-CM}. Consider the convex minorant of $\ti f$ on the interval $[d-g,T]$ and note that $d$ is the right end of the last face whose slope is less than $\ti f(d-g)/(d-g)$. Thus we may identify $d$, $u$ and $g$ and then invert the 3214 transform to recover $f$. This shows that $(u,f)\mapsto (d-g,\Theta_{g,u,d}f)$ is injective, completing the proof.
\end{proof}

\subsection{Proof of 
%Theorems~\ref{thm:SB-representation} 
Theorem~\ref{thm:levy-minorant}}
\label{subsec:proof-main-thm}
%\begin{proof}[Proof of Theorem~\ref{thm:levy-minorant}]

\textbf{Step 1.} Let $\wt C_k$ be the largest convex function on $[0,T]$ that is smaller than $X$ pointwise on the set $D_k\coloneqq\{Tn/2^k:n\in\{0,1,\ldots,2^k\}\}$. Since $D_{k}\subset D_{k+1}$, we have $\wt C_k(t)\geq \wt C_{k+1}(t)$ for all $t\in[0,T]$.  Moreover, the limit $\wt C_\infty\coloneqq\lim_{k\to\infty}\wt C_k$ is clearly convex and smaller than $X$ pointwise on the dense set $\bigcup_{k\in\N}D_k $ in $[0,T]$. As $X$ is \cadlag, $\wt C_\infty$ is pointwise smaller than $X$ on $[0,T]$, implying $\wt C_\infty\le\CM{T}{X}$. Since $\CM{T}{X}$ is convex and smaller than $X$ on $D_k$, the maximality of $\wt C_k$ yields $\wt C_k\ge\CM{T}{X}$ for all $k\in\N$, implying $\wt C_\infty\ge \CM{T}{X}$ and thus $\wt C_\infty=\CM{T}{X}$. 

\textbf{Step 2.} Let $U_1,U_2,\ldots$ be an iid sequence of $\U(0,1)$ random variables independent of $X$. Let $L_0\coloneqq T$, $L_n\coloneqq U_nL_{n-1}$, $\ell_n\coloneqq L_{n-1}-L_n$ and $\xi_n\coloneqq X_{L_{n-1}}-X_{L_n}$ for $n\in\N$. For each $k\in\N$, define $L_{k,0}\coloneqq T$, $L_{k,n}\coloneqq \fl{L_{k,n-1}U_n2^k/T}T/2^k$, $\ell_{k,n}\coloneqq L_{k,n-1}-L_{k,n}$ and $\xi_{k,n}\coloneqq X_{L_{k,n-1}}-X_{L_{k,n}}$ for $n\in\N$. Let $N_k$ be the largest natural number for which $\ell_{k,N_k}>0$, so that $\ell_{k,n}$, $L_{k,n}$ and $\xi_{k,n}$ are all zero for all $n>N_k$. For each $k\in\N$, let $C_k$ (resp. $C_\infty$) be the piecewise linear convex function given in~\eqref{eq:CM-formula} with lengths $(\ell_{k,n})_{n=1}^{N_k}$ (resp. $(\ell_n)_{n=1}^\infty$) and heights $(\xi_{k,n})_{n=1}^{N_k}$ (resp. $(\xi_n)_{n=1}^\infty$). Next we show that $\|C_{k}-C_\infty\|_\infty\cip 0$ as $k\to\infty$. 
Since $X$ has \cadlag~paths with countably many jumps, $L_n$ has a density for every $n\in\N$ and $L_{k,n}\to L_n$ a.s. as $k\to\infty$, we have  $\xi_{k,n}\to\xi_n$ a.s. as $k\to\infty$ for all $n\in\N$. Thus, by Proposition~\ref{prop:conv_pw_linear}, it suffices to prove that $\lim_{M\to\infty}\limsup_{k\to\infty}\e[1\wedge P_{k,M}]=0$, where $P_{k,M}\coloneqq \sum_{n=M}^{N_k}|\xi_{k_m,n}|$.

Theorem~\ref{thm:SB-representation-RW} implies that $C_k\eqd\wt C_k$. Let $R_k$ be the continuous piecewise linear function connecting the skeleton of $X$ on $D_k$ with line segments. Since the minimum and the final value of the convex minorant $\CM{T}{R_k}=\wt C_k$ agree with the corresponding functionals of $R_k$, the total variation $\sum_{n=1}^{N_k}|\xi_{k,n}|$ of $C_k$ has the same distribution as $X_T-2\min_{t\in D_k}X_t$. Moreover, by the independence and the definition of $(L_{k,n})_{n\in\N}$, it is easily seen that $P_{k,M}=\sum_{n=M+1}^{N_k}|\xi_{k,n}|\eqd X_{L_{k,M}}-2\min_{t\in D_k\cap [0,L_{k,M}]}X_t$. By the inequality $L_{k,M}\le L_M$, we have 
\[
X_{L_{k,M}}-2\min_{t\in D_k\cap [0,L_{k,M}]}X_t
\le X_{L_{k,M}}-2\un{X}_{L_{k,M}}
\le \ov X_{L_M}-2\un{X}_{L_M}.
\]
Since $L_M\to 0$ a.s. as $M\to\infty$ and $\ov X_t-2\un{X}_t\to 0$ a.s. as $t\to 0$, we have $\ov X_{L_M}-2\un{X}_{L_M}\to 0$ a.s. as $M\to\infty$, implying
\[
\limsup_{k\to\infty}
    \e[1\wedge P_{k,M}] 
\le\e[1\wedge (\ov X_{L_M}-2\un{X}_{L_M})]
\xrightarrow[M\to\infty]{} 0.\qedhere
\]

\textbf{Step 3.} 
Recall that, by Theorem~\ref{thm:SB-representation-RW}, we have $C_k\eqd\wt C_k$. Since $\|C_k-C_\infty\|_\infty\cip 0$ and $\|\wt C_k-\CM{T}{X}\|_\infty\to 0$ a.s. as $k\to\infty$, we conclude that $C_\infty\eqd\CM{T}{X}$, implying Theorem~\ref{thm:levy-minorant}.
%\end{proof}

%\begin{proof}[Proof of %Theorem~\ref{thm:SB-representation}]
%Since the L\'evy process $X$ does not %jump at time $T$ a.s., we have %$X_T=Y_T$, where %$Y\coloneqq-\CM{T}{-X}$ is the %concave majorant of $X$, i.e., the %smallest concave function that is %pointwise larger than $X$. The %suprema of $X$ and $Y$ and the times %at which these suprema are attained %agree. Moreover, these functionals of %$Y$ are located after adding all the %faces of $Y$ with positive height. %Thus, Theorem~\ref{thm:levy-minorant} %implies Theorem~\ref{thm:SB-represent%ation}.
%\end{proof}

\appendix

\section{Sticks on an exponential interval are a Poisson point processes}
\label{app:PPP_expon-time}

For $n\geq2$, the Dirichlet law
on the simplex 
$\{(x_1,\ldots,x_n)\in(0,1]^n\,:\, \sum_{i=1}^nx_i=1\}$
with parameters $\theta_i>0$
has a density proportional to 
$(x_1,\ldots,x_n)\mapsto %B(\theta_1,\ldots,\theta_n)^{-1}
\prod_{i=1}^n x_i^{\theta_i-1}$. 
%where
%$\Gamma(\cdot)$ is the gamma function,
%$B(\theta_1,\ldots,\theta_n):=\Gamma(\sum_{i=1}^n\%theta_i)^{-1}\prod_{i=1}^n\Gamma(\theta_i)$ is the %beta function and the parameters are positive.
$D$ is a Dirichlet random measure  on $(0,1]$ 
if for any
$0=t_0<t_1<\ldots<t_n=1$,
%every 
%measurable partition $B_1,\ldots,B_n$ of $(0,1)$,
the random vector
$(D((t_0,t_1]),\ldots,D((t_{n-1},t_n]))$
follows the Dirichlet law with parameters 
$(t_i-t_{i-1})$.
%%\sim \Dir(\mu(B_1),\ldots,\mu(B_n))$ 
%for any 
Let $(U_n)_{n\in\N}$ and $(V_n)_{n\in\N}$ 
be independent iid $\U(0,1)$ sequences, independent of 
a  Dirichlet random measure $D_0$ on $(0,1]$. Elementary calculations imply that $D_1\coloneqq (1-V_1)\delta_{U_1} + V_1 D_0\eqd D_0$ and hence $D_n\coloneqq (1-V_n)\delta_{U_n} + V_n D_{n-1}\eqd D_0$ for all $n\in\N$. Since  $D_n$ converges to $D_\infty:=\sum_{n\in\N}\ell_n\delta_{U_n}$ in total variation, where %$(\ell_n)_{n\in\N}$
$\ell_n\coloneqq (1-V_n)\prod_{k=1}^{n-1} V_k$ is a uniform stick-breaking process on $[0,1]$, we have $D_0\eqd D_\infty$. Moreover, by construction we have 
$\sum_{n\in\N}(\ell_n^{(\theta)}/T_\theta)\delta_{U_n}\eqd D_\infty$, where  
$(\ell_n^{(\theta)})_{n\in\N}$
is a stick-breaking process 
on an independent exponential time horizon
$T_\theta\sim \Exp(\theta)$.

%arrive at Sethuraman's %representation~\cite{MR1309433}: 

Let $G$ be a gamma subordinator (i.e. $G_t$ has density proportional to $s\mapsto s^{t-1}e^{-\theta s}$). 
%no drift and L\'evy measure $t^{-1}e^{-\theta t}\D t$). 
The jump of $G$ at $t>0$, $\Delta G_t:=G_t-\lim_{s\uparrow t}G_s$, is zero for all but countably many $t$, making
$D':= \sum_{t\in(0,1]}\left((\Delta G_t)/G_1\right) \delta_{t}$
a Dirichlet random measure on $(0,1]$, independent of
$G_1\sim \Exp(\theta)$. Indeed,
%Let $(U_n)_{n\in\N}$ and $(\Delta_n)_{n\in\N}$ be %the jump times and sizes, respectively, of $G$ on %the interval $(0,1]$. Then $(U_n)_{n\in\N}$ are %iid $\U(0,1)$ independent of %$(\Delta_n)_{n\in\N}$. It can be easily shown that %$D'\coloneqq\sum_{n\in\N}(\Delta_n/G_1)\delta_{U_n%}\eqd D_0$. Indeed, 
for any $0=t_0<t_1<\ldots<t_n=1$, the Jacobian change-of-variable formula shows that the vector  
\[
\big(D'((t_0,t_1]),\ldots,D'((t_{n-1},t_n]),G_1\big)
=((G_{t_1}-G_{t_0})/G_1,\ldots,(G_{t_n}-G_{t_{n-1}})/G_1,G_1)
\]
has the desired law.
%is $\Dir(t_1-t_0,\ldots,t_n-t_{n-1})\times\Exp(\th%eta)$. Thus, $D'$ is a Dirichlet process with mean %measure $\U(0,1)$ independent of %$G_1\sim\Exp(\theta)$. 
Thus
$(D',G_1)\eqd (\sum_{n\in\N}(\ell_n^{(\theta)}/T_\theta)\delta_{U_n}, T_\theta)$,
implying that the law of the Poisson point process
$\sum_{t\in(0,1]} \1_{\Delta G_t>0} \delta_{\Delta G_t}$
coincides with that of the random measure 
$\sum_{n\in\N}\delta_{\ell_n^{(\theta)}}$.

\bibliographystyle{amsalpha}
\bibliography{References}

\section*{Acknowledgements}

\thanks{
\noindent 
JGC and AM are supported by EPSRC grant EP/V009478/1 and The Alan Turing Institute under the EPSRC grant EP/N510129/1; 
AM was supported by EPSRC grant EP/P003818/2 and the Turing Fellowship funded by the Programme on Data-Centric Engineering of Lloyd's Register Foundation;
JGC was supported by CoNaCyT scholarship 2018-000009-01EXTF-00624 CVU 699336.}

\end{document}